\documentclass[10pt]{article}
\usepackage{amsmath,amsthm,amsfonts}
\usepackage{float}
\usepackage{graphicx}

\usepackage[ruled]{algorithm2e}
\usepackage{amsmath,amsthm,amsfonts,amscd}
\usepackage{eucal} 	 	
\usepackage{verbatim}      	
\usepackage{makeidx}       	
\usepackage{url}		

\def\qed{\hfill{$\vcenter{\hrule height1pt \hbox{\vrule width1pt height5pt
    \kern5pt \vrule width1pt} \hrule height1pt}$} \medskip}
\newtheorem{theorem}{Theorem}
\newtheorem{lemma}[theorem]{Lemma}
\newtheorem{corollary}[theorem]{Corollary}

\newtheorem{example}{Example}

\def\reals{{\rm\vrule depth0ex width.4pt\kern-.08em R}}

\newcommand{\bs}{\bigskip}

\newcommand{\vmin}{\mbox{vmin}}

\theoremstyle{definition}

\SetKwIF{step}{}{}{}{}{}{}{}
\SetKw{KwStepOne}{step 1}
\SetKw{KwStepTwo}{step 2}
\SetKw{KwStepThree}{step 3}
\SetKw{KwStepFour}{step 4}
\newcommand{\latexe}{{\LaTeX\kern.125em2%
                      \lower.5ex\hbox{$\varepsilon$}}}

\chardef\bslash=`\\	

\makeatletter		
\def\square{\RIfM@\bgroup\else$\bgroup\aftergroup$\fi
  \vcenter{\hrule\hbox{\vrule\@height.6em\kern.6em\vrule}%
                                              \hrule}\egroup}
\makeatother		

\makeindex    

\begin{document}

\baselineskip0.24in

\begin{center}
\begin{Large}
\begin{bf}

Staffing Large Service Systems Under Arrival-rate Uncertainty \bs


\end{bf}
\end{Large}
\end{center}

\begin{center}
\textbf{Jing Zan} \\
Zilliant, Inc. \\
Austin, Texas, 78701 \\

\vspace{0.1in}

\textbf{John J. Hasenbein and David Morton} \\
 Graduate Program in Operations Research and Industrial Engineering \\
 Department of Mechanical Engineering \\
 University of Texas at Austin, Austin, Texas, 78712 \\
 
 \vspace{0.1in}
 
 \emph{jingzan@gmail.com, jhas@mail.utexas.edu, morton@mail.utexas.edu}
 \end{center}

\begin{center}
\section*{\large{Abstract}}
\end{center}
We consider the problem of staffing large-scale service systems with multiple
customer classes and multiple dedicated server pools under joint quality-of-service
(QoS) constraints. We first analyze the case in which arrival rates are deterministic and
the QoS metric is the probability a customer is queued, given by the Erlang-C formula. 
We use the Janssen-Van Leeuwaarden-Zwart bounds to obtain asymptotically optimal solutions to
this problem.  The second model considered is one in which the arrival rates are not completely
known in advance (before the server staffing levels are chosen), but rather are known via
a probability distribution. In this case, we provide asymptotically optimal solutions to the 
resulting stochastic integer program, leveraging results obtained for the deterministic
arrivals case. 

\vspace{0.2in}

\section{Introduction} 

In this paper, we consider the problem of staffing large-scale service systems with 
multiple customer classes and dedicated server pools for each class. Staffing is such systems 
must typically be done under quality-of-service (QoS) constraints and in our model the QoS 
metric is the probability that a customer is queued (i.e., he must wait for service). In the
classical $M/M/n$ model, this probability is given by the Erlang-C formula. Since
we analyze multiple server systems, we formulate joint QoS constraints that connect 
the performance of all the server pools. Our first results pertain to such systems when
the arrival rates are deterministic. However, our final goal is to examine systems
in which the arrival rates are  known to the system manager only through the
joint distribution of the arrival rate vector. In this case, the joint QoS constraints become
more interesting, since optimizing the staffing level under such constraints can rely crucially
on the correlations between arrival rates for different classes of customers. 
The first primary contribution of the paper is to show that the Janssen-Van Leeuwaarden-Zwart
bounds \cite{janssen_08} converge uniformly, under Halfin-Whitt-type scaling, to the Erlang-C formula
(see Theorem \ref{UniConvOfBounds}). The other main contribution is to introduce and 
solve problems in which the \emph{correlation} between uncertain arrival rates plays
an important role in staffing problems with QoS metrics. 

The literature on staffing service systems goes back to Erlang himself, in the development of the
Erlang-C formula. More recent work has focused on developing approximations of various 
kinds and even this literature is quite voluminous. In terms of asymptotic approximations with
deterministic system parameters, the fundamental inspiration for our work is Halfin and
Whitt \cite{Halfin_81} which introduced the idea of a many-server limit. More closely related
recent work is that of Borst et al.\ \cite{borst_04}. They consider many-server asymptotic approximations and use the Halfin-Whitt formula to approximately solve both the constrained
and ``dualized'' formulations of their model (compare to equations (\ref{eq:bicriteria_beta}) and 
(\ref{eq:ConstraintModel})
 in Section \ref{sec:model_formulation}). A primary difference in this paper is that we use uniform
convergence results to prove a stronger form of asymptotic convergence. In particular, they
prove that the ratio of the approximate and exact costs converge to unity whereas we demonstrate 
that the difference of the costs go to zero. Furthermore, we also
consider the multi-station case. 
The work in Janssen et al.\ \cite{janssen_08} is also closely related to our results and in fact
we use the bounds developed there to formulate our approximations. They consider the same
single-station staffing problems as described here, obtaining
stronger asymptotic optimality results than \cite{borst_04} via an approximation that
refines the Halfin-Whitt formula. However, they also focus on the
single-station case. 

Until recently, much of the literature on service system staffing focused on the models
in which the systems parameters are known with certainty. However, with renewed interest
in call center modeling, there have been significant efforts to incorporate parameter uncertainty,
especially with respect to arrival rates. Harrison and Zeevi \cite{harrison_05} were probably the first
to explicitly consider multi-station staffing models with joint arrival rate uncertainty. 
They propose a staffing method based on a stochastic fluid approximation, but do not
provide analytical results regarding the accuracy of this approximation. A follow-up
paper, Bassamboo et al.~\cite{bhz06} provides the rigorous justification for the methods proposed in \cite{harrison_05}. 
Bassamboo and Zeevi \cite{bassamboo_07} present a data-driven version of the model considered
in the two aforementioned papers. 
Bassamboo et al.~\cite{brz10} continue this stream of papers by performing a more
nuanced investigation of news-vendor type solutions derived in \cite{harrison_05},
in the single-station case. 

Whitt \cite{whi06} uses a different fluid model formulation to analyze a single-station system
with both arrival rate uncertainty and staffing uncertainty due to absenteeism. 
Gurvich et al.~\cite{gurvich_10} investigate staffing a call center with multiple classes,
QoS constraints, and uncertain arrival rates. However, the QoS requirement in their
model are formulated via chance-constraints which insure that the requirements are
met with high probability. Their paper contains a nice discussion of chance-constrained
versus average QoS performance formulations. Finally, Kocaga et al.~\cite{kaw13}
introduce a single-station
staffing and admissions control model with uncertain arrival rates. Their solution is to use
a diffusion approximation to derive a square-root staffing rule in conjunction with a
threshold-based admissions policy. 

The rest of the paper is structured as follows. In Section \ref{chapter-mathematical-background}
we introduce the model and review the relevant mathematical background. 
The analysis of deterministic arrival rate problems is given in Section \ref{chapter-deterministic-arrival-rates}. In Section \ref{chapter-stochastic-arrival-rates}, we consider the random arrival rate case. 

In the sequel, we use the notation $\mathbb{N}= \{1,2,\ldots\}$ and $\mathbb{Z}_{+}=\{0,1,2,\ldots\}$. As usual, 
$\phi(\cdot)$ is the probability density function (PDF) of a standard normal distribution and $\Phi(\cdot)$ is
the corresponding cumulative distribution function (CDF).

\section{Model Formulations and Mathematical Background}
\label{chapter-mathematical-background}

The basis for all of our models is the classic Erlang-C model, along with related approximations. 
For completeness we first review these standard results. This review also contains a key new result
which shows that the bounds given in Janssen et al.~\cite{janssen_08} actually converge uniformly
in the QED regime to the Erlang-C formula. Numerical evidence indicates that the same type of
result holds for the Halfin-Whitt approximation. However, proving this convergence seems
more difficult and the upper bound given in \cite{janssen_08} is more useful, since it can be used
to guarantee feasibility in problems with QoS constraints. 

\subsection{The $M/M/n$ Queue}
\label{sec:erlang-c}

We consider a standard $M/M/n$ queue operating under the first-come-first-served (FCFS)
service discipline. The arrival rate is denoted by $\lambda$ and the service rate by $\mu$.
Without loss of generality, we assume $\mu=1$. Of course in this case the traffic intensity is
also equal to $\lambda$. 
Let $Q$ be the stationary system size. 
A classical result is that $Q$ has the following distribution:
\begin{equation*}
\mathbb{P}\{Q=k\}=\left\{
\begin{array}{l l}
  \eta\frac{\lambda^{k}}{k!} & \quad \text{for $k = 0, 1, 2, \ldots, n-1$}\\
  \eta\frac{n^n (\lambda/n)^{k}}{n!} & \quad \text{$k = n, n+1, \ldots,$}\\
\end{array} \right.
\label{ErlangC_1}
\end{equation*}
where $\eta$ is a normalizing constant:
\begin{equation*}
\label{ErlangC_2}
\eta = \left[\sum_{k=0}^{n-1}\frac{(\lambda)^k}{k!}+\frac{(\lambda)^n}{n!(1-\lambda/n)}\right]^{-1}.
\end{equation*}

The result above allows us to compute $\mathbb{P}\{Q \ge n\}$, which by PASTA is equal to the stationary probability that a customer waits to receive service. This leads immediately to the Erlang-C formula, which we express
as a function of the number of servers and the arrival rate:
\begin{equation*}
\label{ErlangCFormula}
\alpha(n,\lambda) := \eta\frac{(\lambda)^n}{n!(1-\lambda/n)}.
\end{equation*}
In order to facilitate the analysis in later sections, we also use the Jagers-Van Doorn \cite{jagers_vandoorn_86}
continuous extension of the Erlang-C formula:
\begin{equation}
\label{ContErlangC}
\bar{\alpha}(n,\lambda):=\left[\lambda\int_{0}^{\infty}t e^{-\lambda t} (1+t)^{n-1}dt\right]^{-1},
\end{equation}
where $n$ now is any non-negative real number. Hence, $\bar{\alpha}: \reals^2_+ \rightarrow [0,1]$.

In later sections, we use asymptotic analysis of the $M/M/n$ queue
in which the arrival rate and the number of servers both grow large. 
When these quantities grow large together in a specific manner, this is referred to 
as the Halfin-Whitt regime, due to the results below. 

Halfin and Whitt \cite{Halfin_81} consider a sequence of $M/M/n$ queues, indexed by the number of servers, $n$.
The arrival rate in system $n$ is denoted by $\lambda_n$. 
As $n$ increases the scaling of $\lambda_n$ is such that the traffic intensity $\rho_{n}:=\frac{\lambda_{n}}{n}$ approaches 1 (recall $\mu=1$).  For completeness, we now restate their classic result.

\begin{theorem}
\emph{(Halfin and Whitt \cite{Halfin_81})}
\label{HWtheorem}
Consider a sequence of $M/M/n$ queues with arrival rates $\lambda_n$, $n = 1,2, \ldots.$ 
As $n \rightarrow \infty$,  $\alpha(n, \lambda)$ converges to a constant $\alpha$ with $0<\alpha<1$ if and only if
\begin{equation}
\label{HW_1}
\sqrt{n}(1-\rho_{n})\rightarrow\beta
\end{equation}
for some $\beta > 0$. If \eqref{HW_1} holds, then
\begin{equation}
\label{HW_2}
\alpha=\frac{1}{1+\sqrt{2\pi}\beta\Phi(\beta)e^{\beta^{2}/2}}.
\end{equation}
\end{theorem}
Theorem \ref{HWtheorem} implies that when the system is large enough \eqref{HW_2}, which is called the Halfin-Whitt approximation, approximates the Erlang-C formula well.

In \cite{janssen_08}, Janssen et al.\ provide new bounds for the Erlang-C formula which turn
out to be more analytically tractable than the Halfin-Whitt approximation. In subsequent sections, we use these bounds to build approximate staffing models
and obtain asymptotically optimal solutions.
Hereafter, we refer to these bounds as the JVLZ bounds. 

\begin{theorem}
\emph{(Janssen et al.\ \cite{janssen_08})}
\label{BoundsForErlangC}
Let $\rho=\lambda/ n$ and
\begin{equation}
\label{eq:alpha}
a=\sqrt{-2n(1-\rho+\ln\rho)},
\end{equation}
\begin{equation}
\label{eq:betas}
\beta=(n-\lambda)/\sqrt{\lambda},
\end{equation}
\begin{equation}
\label{eq:gamma}
\gamma=(n-\lambda)/\sqrt{n}=\beta \sqrt{\rho}.
\end{equation}
For $n>\lambda$,
\begin{equation}
\label{eq:UB}
\bar{\alpha} (n,\lambda)\leq\left[\rho+\gamma\left(\frac{\Phi(a)}{\phi(a)}+\frac{2}{3\sqrt{n}}\right)\right]^{-1},
\end{equation}
and
\begin{equation}
\label{eq:LB}
\bar{\alpha} (n,\lambda)\geq\left[\rho+\gamma\left(\frac{\Phi(a)}{\phi(a)}+\frac{2}{3\sqrt{n}+\frac{1}{\phi(a)}\frac{1}{12n-1}}\right)\right]^{-1}.
\end{equation}
\end{theorem}
Using equation \eqref{eq:betas} we can express the continuous Erlang-C formula and its bounds in terms of $\beta$ and $\lambda$.
Define $$\tilde{\alpha}(\beta,\lambda)=\bar{\alpha} (\lambda+\beta\sqrt{\lambda},\lambda)$$ as the continuous Erlang-C formula with respect to $\beta$, where $\bar{\alpha}(\cdot,\cdot)$ is defined in equation \eqref{ContErlangC}. 
Notice that the formula above implies that for a particular arrival rate $\lambda$, that a \emph{square-root safety} 
staffing level 
$\lambda+\beta\sqrt{\lambda}$ is used. 

Now, let $UB(\beta,\lambda)$ represent the upper bound on the Erlang-C, as given on the right-hand side of inequality \eqref{eq:UB}, and let $LB(\beta,\lambda)$ represent the lower bound, as given on the right-hand side of inequality \eqref{eq:LB}. One then might hope that if $\lambda$ grows large and square-root safety staffing is used, 
then the upper and lower bounds converge to the Erlang-C formula. In fact, this is already known
to be true via the Halfin-Whitt result. What we demonstrate here is that the convergence is actually
uniform in the square-root staffing factor $\beta$. This uniform convergence is useful for proving asymptotic approximation
results for the optimization problems we formulate in later sections. 

Before we present the main result we state a series of lemmas needed to establish
the primary convergence result of this section. The proofs of the next five lemmas
are given in the Appendix. 
\begin{lemma}
\label{UBDecrease}
For positive $\lambda$, $UB(\beta,\lambda)$  is strictly decreasing in $\lambda$ for any fixed $\beta>0$.
\end{lemma}

\begin{lemma}
\label{1}
Let $n=\lambda+\beta\sqrt{\lambda}$ and let $\gamma$ be defined as in \eqref{eq:gamma}. Then $\gamma/\left(12n-1\right)$ converges uniformly to 0, in $\beta$,
as $\lambda\rightarrow\infty$. That is $$\lim_{\lambda\rightarrow\infty}\sup_{\beta>0}\frac{\gamma}{12n-1}=0.$$
\end{lemma}

\begin{lemma}
\label{3}
Let $n=\lambda+\beta\sqrt{\lambda}$ and let $\rho$, $\gamma$, and $a$ be defined as in Theorem \ref{BoundsForErlangC}. Then $\rho\phi(a)+\gamma\Phi(a)$ is strictly increasing in $\beta$ for any sufficiently large $\lambda$.
\end{lemma}

\begin{lemma}
\label{2}
Let $n=\lambda+\beta\sqrt{\lambda}$ and let $\rho$, $\gamma$, and $a$ be defined as in Theorem \ref{BoundsForErlangC}. Then $\rho\phi(a)+\gamma\Phi(a)+\frac{2\gamma\phi(a)}{3\sqrt{n}}+\frac{\gamma}{(12n-1)}$ is
uniformly bounded away from 0 for all sufficiently large $\lambda$, and all $\beta>0$. Specifically,
\begin{equation}
\label{211}
\inf_{\lambda\geq M, \beta>0}\rho\phi(a)+\gamma\Phi(a)+\frac{2\gamma\phi(a)}{3\sqrt{n}}+\frac{\gamma}{(12n-1)}>0,
\end{equation}
for $M$ sufficiently large.
\end{lemma}

\begin{lemma}
\label{UBDecreaseInBeta}
 For positive $\beta$, $UB({\beta,\lambda})$ is strictly decreasing in $\beta$ for any fixed $\lambda>0$.
\end{lemma}

Finally, we are ready to state and prove the main convergence result of this section,
Theorem \ref{UniConvOfBounds}. This result plays a prominent role in proving the
asymptotic optimality results in Sections~\ref{chapter-deterministic-arrival-rates} and \ref{chapter-stochastic-arrival-rates}.

\begin{theorem}
\label{UniConvOfBounds}
The upper bound of \eqref{eq:UB} and lower bound of \eqref{eq:LB} converge uniformly, over positive $\beta$,
 to the Erlang-C formula as $\lambda\rightarrow\infty$. That is $$\lim_{\lambda\rightarrow\infty}\sup_{\beta>0}\left[\tilde{\alpha}(\beta,\lambda)-LB(\beta,\lambda)\right]=0$$ and $$\lim_{\lambda\rightarrow\infty}\sup_{\beta>0}\left[UB(\beta,\lambda)-\tilde{\alpha}(\beta,\lambda)\right]=0.$$
\end{theorem}
\begin{proof}
\label{UniConvOfBoundspf}
To prove the uniform convergence, we only need to show that the upper bound, $UB(\beta,\lambda)$, uniformly converges to the lower bound, $LB(\beta,\lambda)$, i.e., $$\lim_{\lambda\rightarrow\infty}\sup_{\beta>0}\left[UB(\beta,\lambda)-LB(\beta,\lambda)\right]=0.$$ With $n=\lambda+\beta\sqrt{\lambda}$, we have
\begin{eqnarray*}
UB(\beta,\lambda)-LB(\beta,\lambda) & = & \frac{1}{\rho+\gamma\left(\frac{\Phi(a)}{\phi(a)}+\frac{2}{3\sqrt{n}}\right)}-\frac{1}{\rho+\gamma\left(\frac{\Phi(a)}{\phi(a)}+\frac{2}{3\sqrt{n}}\right)+\frac{\gamma}{\phi(a)\left(12n-1\right)}}\\
& = & \frac{1}{\rho+\gamma\left(\frac{\Phi(a)}{\phi(a)}+\frac{2}{3\sqrt{n}}\right)}\cdot\frac{\frac{\gamma}{\phi(a)\left(12n-1\right)}}{\rho+\frac{\gamma\Phi(a)}{\phi(a)}+\frac{2\gamma}{3\sqrt{n}}+\frac{\gamma}{\phi(a)\left(12n-1\right)}}.
\end{eqnarray*}
From Lemma \ref{UBDecreaseInBeta}, we have that $$\frac{1}{\rho+\gamma\left(\frac{\Phi(a)}{\phi(a)}+\frac{2}{3\sqrt{n}}\right)}$$ is strictly decreasing in $\beta$ for any
fixed $\lambda>0$. Furthermore, it is easy to verify that $$\left.\left[\rho+\gamma\left(\frac{\Phi(a)}{\phi(a)}+\frac{2}{3\sqrt{n}}\right)\right]^{-1} \right|_{\beta =0}=1 \quad\forall\lambda>0.$$ Thus
to show $$\lim_{\lambda\rightarrow\infty}\sup_{\beta>0}UB(\beta,\lambda)-LB(\beta,\lambda)=0,$$ it suffices to prove $$\lim_{\lambda\rightarrow\infty}\sup_{\beta>0}\frac{\gamma}{(12n-1)}=0,$$ and $$\inf_{\lambda\geq 1, \beta>0}\phi(a)\left(\rho+\frac{\gamma\Phi(a)}{\phi(a)}+\frac{2\gamma}{3\sqrt{n}}+\frac{\gamma}{\phi(a)(12n-1)}\right)>0.$$ These conditions hold by Lemma \ref{1} and Lemma \ref{2}, respectively.
\end{proof}


\subsection{Model Formulation}
\label{sec:model_formulation}

We now introduce the single- and multi-station models of interest in this paper. 
The manager of these systems is concerned with both staffing costs and quality
of service. In this paper, we use the probability that a customer must wait to receive service to measure the quality of service and model the trade-off between the staffing cost and this probability. 
We first consider a service center modeled by a single $M/M/n$ queue and then examine
systems with $L$ parallel $M/M/n$ queues, as depicted in Figure \ref{MultiQueueSystemDetPic}. 
\begin{figure}[htp]
\begin{center}
\includegraphics[scale=0.4]{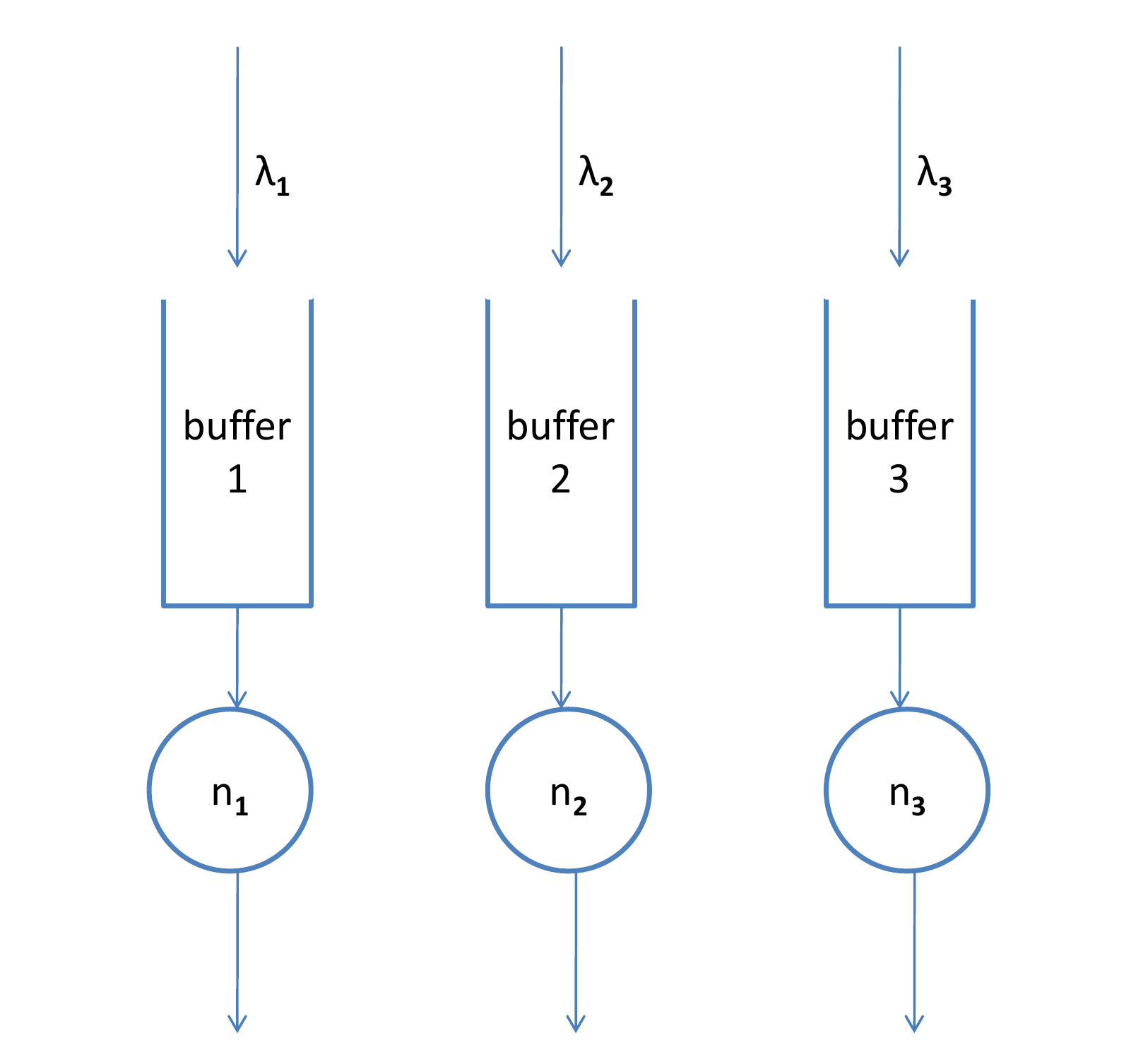}
\caption{Multi-station System}
\label{MultiQueueSystemDetPic}
\end{center}
\end{figure}

We begin with the single-station system. Because of the manager's competing measures, we face a bi-criteria optimization problem, in which we want to simultaneously minimize the staffing cost and the probability of inducing customer waiting. Let $\bar{c}(n)$ be the staffing cost function, and assume $\bar{c}(n)$ is strictly increasing in the staffing level $n$. Next, for an $M/M/n$ queue with arrival rate $\lambda$, let $W(n,\lambda)$ be a random variable
corresponding to the stationary delay (waiting for service). Our bi-criteria model for this call center problem is:
\begin{equation}
\label{eq:bicriteria}
\vmin_{n \in \mathbb{Z}_{+}}\quad \left[\bar{c}(n), \mathbb{P}\left\{ W(n,\lambda)> 0\right\}\right],
\end{equation}
where ``\vmin'' denotes vector minimization and a solution of model \eqref{eq:bicriteria} corresponds to the family of staffing levels that falls on the efficient frontier.
Of course, $\mathbb{P}\left\{ W(n,\lambda)> 0\right\}$ is equal to the steady-state
(or long-run) probability that a customer must wait for service. By PASTA this is 
equal to $\mathbb{P}\left\{ Q(n,\lambda) \ge n \right\}$.

In general, we are interested in asymptotic solutions to this bicriteria model. We consider a sequence of problems of the form \eqref{eq:bicriteria} with $\lambda\rightarrow\infty$. As $\lambda$ goes to $\infty$, the staffing level $n$ also goes to $\infty$ as does the staffing cost $\bar{c}(n)$. Hence, we need to reformulate \eqref{eq:bicriteria} to obtain a well-posed model.  We use the square-root staffing discussed earlier in the context
of the Halfin-Whitt regime. Next, replace the staffing level $n$ in \eqref{eq:bicriteria} with  $\lambda+\beta\sqrt{\lambda}$ and rewrite the model using the decision variable $\beta$:
\begin{equation}
\label{eq:bicriteria_beta}
\vmin_{\beta \geq 0}\quad \left[c(\beta), \mathbb{P}\left\{ W(\beta,\lambda)> 0\right\}\right],
\end{equation}
where $c(\beta)$ is the cost function parameterized in $\beta$ rather than in $n$ and $W(\beta,\lambda)$ 
is again the steady-state delay. Thus, we reformulate the decision problem as one of choosing
the safety-staffing parameter $\beta$, rather than the number of servers. We assume $c(\beta)$ is continuous and strictly increasing in its argument, and hence for any fixed value of $\lambda$, models \eqref{eq:bicriteria} and \eqref{eq:bicriteria_beta} are equivalent. (Recall, our goal in these bi-criteria models is to form the efficient frontier of solutions.) Moreover, the optimal value of $\beta$ does not grow large as $\lambda$ grows large, and hence the asymptotics associated with model \eqref{eq:bicriteria_beta} have finite limits.


One way to solve model \eqref{eq:bicriteria_beta} is to make one component of the objective function a constraint. For example, we can solve the bi-criteria problem by solving a family of models:
\begin{equation}
\min_{\beta \geq 0}\quad c(\beta)\\ \qquad
\mbox{s.t. } \quad \mathbb{P}\left\{ W(\beta,\lambda)> 0\right\}\leq \epsilon,
\label{eq:ConstraintModel}
\end{equation}
parameterized in the risk level threshold, $\epsilon$, where $0<\epsilon<1$.

For each $\epsilon$, by solving model \eqref{eq:ConstraintModel}, we obtain an optimal $\beta$. The staffing cost and the probability of waiting corresponding to the optimal $\beta$ give one point on the efficient frontier of the bi-criteria problem. By varying $\epsilon$ from 0 to 1, we obtain all the points on the efficient frontier.

Another way to solve model \eqref{eq:bicriteria_beta} is to use a weighted objective function approach. That is we solve the bi-criteria problem by solving the family of models
\begin{equation}
\min_{\beta \geq 0}\quad c(\beta)+\delta\mathbb{P}\left\{W(\beta,\lambda)> 0\right\},\\
\label{eq:CostModel}
\end{equation}
parameterized by $\delta>0$, the weight on the second term in the objective function. Solving model \eqref{eq:CostModel}, by varying $\delta$, we obtain all the extreme points of the convex hull of the efficient frontier (cf.\ \cite{Ralphs_06}).

\begin{figure}
\begin{center}
\includegraphics[scale=0.4]{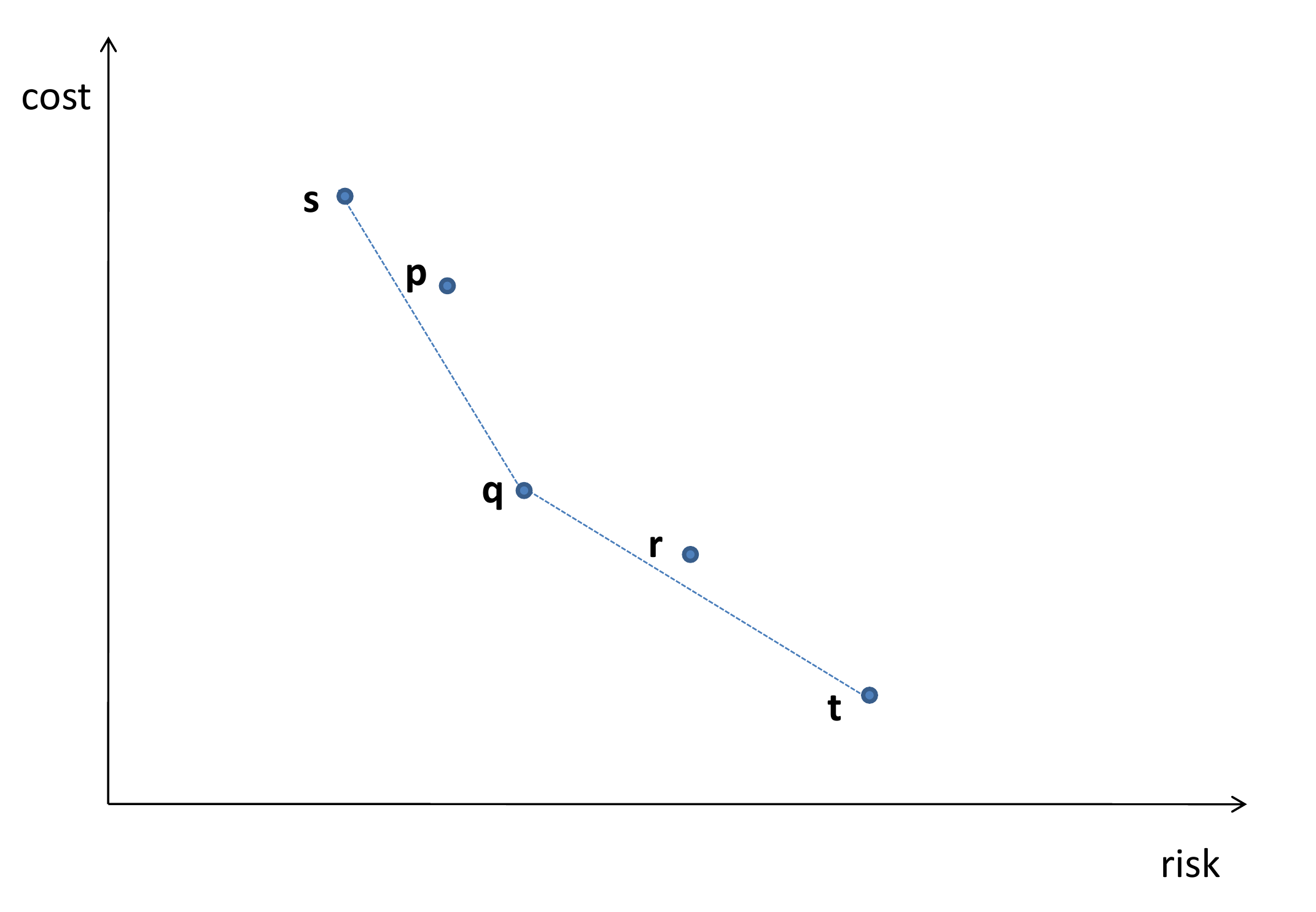}
\caption{An Efficient Frontier}
\label{FrontierPic}
\end{center}
\end{figure}
We use the example given in Figure \ref{FrontierPic} to explain the relationship between model \eqref{eq:ConstraintModel} and model \eqref{eq:CostModel}. In Figure \ref{FrontierPic}, we assume the points $p$, $q$, $r$, $s$ and $t$ correspond to Pareto efficient solutions to the bi-criteria problem under consideration. If we use model \eqref{eq:ConstraintModel} to solve the bi-criteria problem, then by varying $\epsilon$, we achieve all solutions on the efficient frontier. So, we achieve all five points, $p$, $q$, $r$, $s$ and $t$. On the other hand, if we use model \eqref{eq:CostModel}, unless the efficient frontier is convex, we do not achieve all five of these points. However, the solutions which are extreme points of the efficient frontier are achieved. In other words, points $s$, $q$ and $t$ are achieved by solving model \eqref{eq:CostModel} and varying $\delta$.

Model \eqref{eq:ConstraintModel} directly describes what is typically viewed as the practical need. Generally, service center managers try to find a staffing level that minimizes the staffing cost while maintaining a certain service level. However, in our view there is insight to be gained by forming the efficient frontier to better understand cost-quality tradeoffs. This is particularly true when contractual service levels have not yet been determined. In what follows we use either model \eqref{eq:ConstraintModel} or \eqref{eq:CostModel} to present results, depending on which is more convenient.

We now extend our model to a multi-station system as depicted in Figure \ref{MultiQueueSystemDetPic}.  
Suppose we have $L$ $M/M/n$ queues in parallel. 
Station $i$, $i=1,\ldots, L$, has arrival rate $\lambda_{i}$ and $n_{i}$ servers, determined by $\beta_{i}$ via the square-root staffing rule. Then we formulate:
\begin{equation}
\min_{\beta \geq 0}\quad \sum_{i=1}^{L}c_{i}(\beta_{i})+\delta\mathbb{P}\left\{ \bigcup_{i=1}^{L} \left\{W_{i}(\beta_{i},\lambda_{i})> 0 \right\}\right\},\\ \qquad
\label{objmodel1}
\end{equation}
where, $\beta=(\beta_{1},\ldots, \beta_{L})$ and $W_{i}(\beta_{i},\lambda_{i})$
is the stationary delay at station $i$. The second term in the sum above 
again incorporates costs related to quality of service. We assume that $c_i(\beta_i)$ is
continuous and strictly increasing in $\beta_i$ for $i = 1, \ldots, L$ and that the $L$ stations operate independently, i.e., the arrival and service processes are
mutually independent. Then the above model is equivalent to:
\begin{equation}
\min_{\beta \geq 0}\quad \sum_{i=1}^{L}c_{i}(\beta_{i})+\delta\left(1-\prod_{i=1}^{L}\left(1-\mathbb{P}\left\{ W_{i}\left(\beta_{i},\lambda_{i}\right)> 0 \right\}\right)\right).\\ \qquad
\label{eq:MultiObjModel}
\end{equation}

\section{Deterministic Arrival-rate Problems}
\label{chapter-deterministic-arrival-rates}
In this section, we formulate approximate versions of the optimization models presented in
the previous section, using the JVLZ bounds. 
We prove asymptotic optimality of the approximate solutions for single station, and then
multi-station models. These results are primarily stepping stones for the asymptotic
optimality results in Section \ref{chapter-stochastic-arrival-rates}, where we analyze
the case of random arrival rates. 

%
%
\subsection{Single-station System}
We begin with the single-station system and consider model \eqref{eq:ConstraintModel}:
\begin{equation*}
\min_{\beta \geq 0}\quad c(\beta)\\ \qquad
\mbox{s.t. } \quad \mathbb{P}\left\{W(\beta,\lambda)> 0\right\}\leq \epsilon,
\end{equation*}
Recall that the probability in the constraint is given by
$\mathbb{P}\left\{ W(\beta,\lambda)> 0\right\}=\tilde{\alpha}(\beta,\lambda)$.
Next, for any fixed $\lambda$, define model $F_{\lambda}$ as:
\begin{equation}
\min_{\beta \geq 0}\quad c(\beta)\\ \qquad
\mbox{s.t. } \tilde{\alpha}(\beta,\lambda)\leq \epsilon.
\label{eq:1D_original_2}
\end{equation}

As $\lambda$ varies, we obtain a sequence of models $\{F_{\lambda}\}$. The Erlang-C formula can be numerically unwieldy, especially when the arrival rate grows large. So we build an approximate model by replacing the Erlang-C formula with the JVLZ upper bound, $UB({\beta,\lambda}),$ defined by the equation on the right-hand side of \eqref{eq:UB}, except that $n$ is replaced by $\lambda+\beta\sqrt{\lambda}$. Using $UB({\beta,\lambda})$ we define our approximate model $G_{\lambda}$ as:
\begin{equation}
\min_{\beta \geq 0}\quad c(\beta)\\ \qquad
\mbox{s.t. } UB(\beta,\lambda)\leq \epsilon.
\label{eq:1D_modify_3}
\end{equation}
Notice that any feasible solution of model $G_{\lambda}$ is also feasible for model $F_{\lambda}$.

Theorem \ref{Det_Single_optimalityResult} establishes the asymptotic optimality of using
solutions of $G_{\lambda}$ to solve model $F_{\lambda}$. The theorem is stated in
terms of convergence of the decision variables. Since the cost function is assumed to
be continuous, this also implies convergence of the objective values. 
Before turning to Theorem \ref{Det_Single_optimalityResult}, we first provide a supporting lemma.
\begin{lemma}
\label{ErlangCDecreaseInBeta}
Let $\lambda>0$, $\bar{\alpha}({n,\lambda})$ be as defined in \eqref{ContErlangC}, and set $\tilde{\alpha}({\beta,\lambda})=\bar{\alpha}({\lambda+\beta\sqrt{\lambda},\lambda})$. Then $\tilde{\alpha}({\beta,\lambda})$ is strictly decreasing in $\beta$ for any $\lambda, \beta >0$ that satisfy $\lambda+\beta\sqrt{\lambda}\geq1.$
\end{lemma}
\begin{proof}
To prove $\tilde{\alpha}({\beta,\lambda})$ is strictly decreasing in $\beta$, it suffices to show that $\bar{\alpha}({n,\lambda})$ is strictly decreasing in $n$. Jagers and Van Doorn \cite{jagers_vandoorn_86} prove that $\bar{\alpha}({n,\lambda})$ is convex in $n$. Also, we know that $\bar{\alpha}({n,\lambda})$ is strictly decreasing in $n$ on the positive integers. This implies $\bar{\alpha}({n,\lambda})$ is strictly decreasing in $n$ for all $n>1$, otherwise its epigraph is not convex.
\end{proof}

\begin{theorem}
\label{Det_Single_optimalityResult}
For $\lambda>0$, let the optimal solution of $F_\lambda$, as defined in \eqref{eq:1D_original_2}, be $\beta_{\lambda}^{F}$ and let the optimal solution of $G_\lambda$, as defined in \eqref{eq:1D_modify_3}, be $\beta_{\lambda}^{G}$. Then $\beta_{\lambda}^{G}\geq\beta_{\lambda}^{F}$, $\forall \lambda>0$, and there exists a finite $\beta^{*}$ such that $$\lim_{\lambda\rightarrow \infty}\beta_{\lambda}^{G}=\lim_{\lambda\rightarrow \infty}\beta_{\lambda}^{F}=\beta^{*}.$$
\end{theorem}
\begin{proof}
The objective function, $c(\beta)$, is strictly increasing in $\beta$ and by Lemmas \ref{UBDecreaseInBeta} and \ref{ErlangCDecreaseInBeta}, $\tilde{\alpha}(\beta,\lambda)$ and $UB(\beta,\lambda)$ are strictly decreasing and continuous in $\beta$. Hence, the unique optimal solution of models \eqref{eq:1D_original_2} and \eqref{eq:1D_modify_3} are defined by requiring the respective constraints to hold with equality. That is, $\beta_{\lambda}^{F}$ solves $\tilde{\alpha}(\beta,\lambda)= \epsilon$ and $\beta_{\lambda}^{G}$ solves $UB(\beta,\lambda)= \epsilon$. We know that $UB(\beta, \lambda)$ and $\tilde{\alpha}(\beta,\lambda)$ have range (0,1] and we have $$UB(0, \lambda)=\tilde{\alpha}(0,\lambda)=1$$ and $$\lim_{\beta\rightarrow\infty}UB(\beta,\lambda)=\lim_{\beta\rightarrow\infty}\tilde{\alpha}(\beta,\lambda)=0.$$ Also $UB(\beta, \lambda)$ and $\tilde{\alpha}(\beta,\lambda)$ are continuous and strictly decreasing in $\beta$ on $[0, \infty)$. So, the optimal solutions $\beta_{\lambda}^{F}$ and $\beta_{\lambda}^{G}$ exist and are unique for any $\epsilon>0$ and $\lambda>0$. From Lemma \ref{UBDecrease}, we have $\beta_{\lambda_1}^{G}>\beta_{\lambda_2}^{G}\geq 0$, for any $\lambda_2>\lambda_1$.
This indicates $\lim_{\lambda\rightarrow\infty}\beta_{\lambda}^{G}$ exists and is finite. Let $\lim_{\lambda\rightarrow\infty}\beta_{\lambda}^{G}=\beta^{*}$. Since $UB(\beta,\lambda)\geq\tilde{\alpha}(\beta,\lambda)$ for any $\beta>0$ and $\lambda>0$, we have $\beta_{\lambda}^{G}\geq\beta_{\lambda}^{F}$
for any $\lambda>0$. This together with the fact that $\{\beta_{\lambda}^{G}\}$ is a bounded sequence, indicates that $\{\beta_{\lambda}^{F}\}$ is a bounded sequence. So $\{\beta_{\lambda}^{F}\}$ has at least one subsequence that has a finite limit. For any subsequence $\{\beta_{\lambda'}^{F}\}$ with a limit and its corresponding limit $\hat{\beta}$, we have $$\lim_{\lambda'\rightarrow\infty}\tilde{\alpha}(\hat{\beta},\lambda')=\epsilon.$$ Also, for any $\beta$, we have $$\lim_{\lambda\rightarrow\infty}(\tilde{\alpha}(\beta,\lambda)-UB(\beta,\lambda))=0.$$ This indicates that $$\lim_{\lambda'\rightarrow\infty}UB(\hat{\beta},\lambda')=\lim_{\lambda'\rightarrow\infty}\tilde{\alpha}(\hat{\beta},\lambda')=\epsilon.$$ Since $\lim_{\lambda'\rightarrow\infty}UB(\beta^{*},\lambda')=\epsilon,$ and there exists a unique $\beta$ satisfying $\lim_{\lambda'\rightarrow\infty}UB(\beta,\lambda')=\epsilon,$ we have that $\hat{\beta}=\beta^{*}.$ This implies that all subsequences of $\{\beta_{\lambda}^{F}\}$ have the same limit point, $\beta^{*}$. Thus 
$\lim_{\lambda\rightarrow \infty} \beta_{\lambda}^{F}$ exists and is $\beta^{*}$.
In other words, $$\lim_{\lambda\rightarrow \infty}\beta_{\lambda}^{G}=\lim_{\lambda\rightarrow \infty}\beta_{\lambda}^{F}=\beta^{*}.$$
\end{proof}

\subsection{Multi-station System}
\label{sec:mss}
We now extend our development to a multi-station system and return our attention to model (\ref{eq:MultiObjModel}):
\begin{equation*}
\min_{\beta \geq 0}\quad \sum_{i=1}^{L}c_{i}(\beta_{i})+\delta\left(1-\prod_{i=1}^{L}\left(1-\mathbb{P}\left\{ W_{i}\left(\beta_{i},\lambda_{i}\right)> 0 \right\}\right)\right).\\ \qquad
\end{equation*}
%
As in the single-station system, we formulate an equivalent model using the continuous Erlang-C formula, and we again denote this by model $F_{\lambda}$:
\begin{equation}
\min_{\beta \geq 0}\quad \sum_{i=1}^{L}c_{i}(\beta_{i})+\delta\left(1-\prod_{i=1}^{L}\left(1-\tilde{\alpha}\left(\beta_{i},\lambda_{i}\right)\right)\right).\\ \qquad
\label{objmodel_Erlang}
\end{equation}

Following an analogous development to our single-station system, we build an approximate model for \eqref{objmodel_Erlang} by using the JVLZ bound $UB(\beta,\lambda)$. The approximate model $G_{\lambda}$ is:
\begin{equation}
\min_{\beta \geq 0}\quad \sum_{i=1}^{L}c_{i}(\beta_{i})+\delta\left(1-\prod_{i=1}^{L}\left(1-UB\left( \beta_{i},\lambda_{i}\right)\right)\right).\\ \qquad
\label{objmodel_UB}
\end{equation}

Denote the objective function of model $F_{\lambda}$ as $f_{\lambda}(\cdot)$, and the optimal solution of $F_{\lambda}$ as the $L$-vector $\beta_{\lambda}^{F}$. Similarly, denote the objective function of model $G_{\lambda}$ as $g_{\lambda}(\cdot)$, and the optimal solution of $G_{\lambda}$ as $\beta_{\lambda}^{G}$. Theorem \ref{multi_objvalue_converge} implies the asymptotic optimality of solutions to the approximate model as the arrival rate vector grows large. We let the arrival rates grow in the following way. Assume there are initial values of arrival rates for all queues. Let the initial vector of rates be $\lambda^{0}=(\lambda^{0}_{1},\ldots,\lambda^{0}_{L})$. Indexing the sequence of systems under consideration with positive integers, assume the arrival rate for the $m^{th}$ system is $\lambda^{m}=m\lambda^{0}$. Then as $m\rightarrow\infty$ the
components of $\lambda^{m}$ grow large together. 

The lemma below, needed for the main result, is proved in the Appendix. 
\begin{lemma}
\label{Uniform_Converge_Multi_Term}
Let $f_{m}(\cdot)$ denote the objective function of model $F_{\lambda}$ as defined in \eqref{objmodel_Erlang}, with arrival rate $\lambda^m=m\lambda^0$. And, let $g_{m}(\cdot)$ denote the objective function of model $G_{\lambda}$ as defined in \eqref{objmodel_UB}, with arrival rate $\lambda^m$. Then, $$\lim_{m\rightarrow\infty}\sup_{\beta\geq0}\left(g_{m}(\beta)-f_{m}(\beta)\right)=0.$$
\label{UniConvofProduct}
\end{lemma}

\begin{theorem}
\label{multi_objvalue_converge}
Let $f_{m}(\cdot)$ denote the objective function and let $\beta_{m}^{F}$ denote the optimal solution of model $F_{\lambda}$ as defined in \eqref{objmodel_Erlang}, with arrival rate $\lambda^m$. Let $g_{m}(\cdot)$ denote the objective function, and let $\beta_{m}^{G}$ denote the optimal solution of model $G_\lambda$ as defined in \eqref{objmodel_UB}, with arrival rate $\lambda^m$. Then $$\lim_{m\rightarrow\infty}\left(f_{m}(\beta_{m}^{G})-f_{m}(\beta_{m}^{F})\right)=0.$$
\end{theorem}
\begin{proof}
First, since $\beta_m^F$ is optimal with respect to $F_{\lambda}$ we have $f_{m}(\beta_{m}^{G})-f_{m}(\beta_{m}^{F})\geq0$. The objective function $g_m$ is an upper bound for $f_m$ and so we have $g_{m}(\beta_{m}^{G})\geq f_{m}(\beta_{m}^{G})$. Thus $f_{m}(\beta_{m}^{G})-f_{m}(\beta_{m}^{F})\leq g_{m}(\beta_{m}^{G})-f_{m}(\beta_{m}^{F})$.
We also have $g_{m}(\beta_{m}^{F})\geq g_{m}(\beta_{m}^{G})$, since $\beta_{m}^{G}$ is optimal, with respect to $G_{\lambda}$ with arrival rate $\lambda^m$. This implies that
$f_{m}(\beta_{m}^{G})-f_{m}(\beta_{m}^{F})\leq g_{m}(\beta_{m}^{F})-f_{m}(\beta_{m}^{F})$. According to Lemma \ref{Uniform_Converge_Multi_Term}, we have $g_{m}(\beta_{m}^{F})-f_{m}(\beta_{m}^{F})\rightarrow 0$, as $m\rightarrow\infty$. This proves that
$f_{m}(\beta_{m}^{G})-f_{m}(\beta_{m}^{F})\rightarrow 0$, as $m\rightarrow\infty$.
\end{proof}

Theorem \ref{multi_objvalue_converge} indicates that for the multi-station problem we describe above, our approximate solution is asymptotically optimal for any $\delta>0$, in the sense that the absolute gap between the optimal objective value and the approximate objective value goes to 0. Notice that Theorem \ref{Det_Single_optimalityResult},
for the single-station case, is slightly stronger, in that it implies that both the approximate solutions and the
approximate objective values converge to the respective optima. 

We could also build the approximating problem using the lower JVLZ bound, $LB(\beta,\lambda)$, to replace the Erlang-C formula in the original model $F_{\lambda}$. In this case, we can again obtain a result analogous to Theorem \ref{multi_objvalue_converge}. We prefer to employ the upper bound rather than the lower bound because the solution under the former approximation is appropriately conservative, i.e., it is guaranteed to be feasible for model $F_{\lambda}$, while the solution under the lower bound is not. Finally, one
could also build approximating problems using the Halfin-Whitt formula. However, 
the required uniform convergence result seems harder to establish. 

\section{Stochastic Arrival-rate Problems}
\label{chapter-stochastic-arrival-rates}
In this section, we extend models from Section~\ref{chapter-deterministic-arrival-rates} to include arrival-rate uncertainty. That is, we focus on solving large-scale staffing problems when the arrival rates are uncertain in addition to the inherent randomness of the system's inter-arrival times and service times. In particular, we consider a decision making scheme in which the manager must select staffing levels before observing the arrival rates. 
However, the decision maker does have complete distributional knowledge of the
arrival rates. Such a model reflects typical practical situations in which the staffing
schedule must be determined in advance, with only a forecast of possible daily
call volumes in hand. The manager's goal is now to minimize staffing costs while
meeting QoS levels averaged over time. We represent this by taking an expected value
over the possible arrival rates. 

In the setting where the state space of possible arrival rates is discrete, 
we show that as the system size grows, there is at most one {\em key} scenario under which the probability of waiting converges to a non-trivial value, i.e., a value strictly between 0 and 1. In any other scenario, the probability of waiting converges to either 0 or 1, that is the staffing level is either over- or under-loaded in any scenario other than the key scenario. Exploiting this result, we propose a two-step solution procedure for the staffing problem with random arrival rates. In the first step, we use the desired QoS level to identify the key scenario corresponding to the optimal staffing level. After finding the key scenario, the random arrival-rate model reduces to a deterministic arrival-rate model. In the second step, we solve the resulting model, with a deterministic arrival rate, by using our approximation model proposed in Section~\ref{chapter-deterministic-arrival-rates}. The approximate optimal staffing level obtained in this procedure converges to the true optimal staffing level for the random arrival-rate problem as the system's size grows large.

\subsection{Single-station System}
\label{sec:RandomModel}
As before, we first analyze the single-station system and then turn to the
case of $L$ parallel $M/M/n$ queues.
So, in the single-station system, let $\Lambda$ denote the random arrival rate. Let $\Lambda^{\omega}$ be a specific realization, where $\omega$ is an outcome, or \emph{scenario}, from the sample space $\Omega$. We assume that $\Omega$ is finite. Let $p^{\omega}$ be the probability assigned to scenario $\omega$.
A naive attempt to extend \eqref{eq:ConstraintModel} to the doubly stochastic setting
results in the following model:
\begin{equation}
\min_{\beta \geq 0}\quad c(\beta)\\ \qquad
\mbox{s.t. } \sum_{\omega\in\Omega} p^{\omega} \mathbb{P}\left\{ W(\beta,\Lambda)> 0 \mid\Lambda=\Lambda^{\omega}\right\}\leq \epsilon,
\label{eq:Random_Single_Chance2_0}
\end{equation}
where $\mathbb{P}\left\{ W(\beta,\Lambda)> 0 \mid\Lambda=\Lambda^{\omega}\right\}=\tilde{\alpha}(\beta,\Lambda^{\omega})$.

However, there is a fundamental shortcoming in this formulation. In our desired stochastic program the staffing decision made at time 0 should be nonanticipative; i.e., it cannot depend on a realization of the randomness not yet observed. However, in this formulation, the number of servers does depend on $\omega$ in that the number of servers varies by $\omega$ via $\Lambda^\omega + \beta \sqrt{\Lambda^\omega}$. In order to rectify this,
the decision at time 0 must consist of both choosing the square-root staffing factor
$\beta$ \emph{and} a specific scenario $\omega^{key}$. Therefore, the number
of servers chosen is given by $\Lambda^{\omega^{key}} + \beta \sqrt{\Lambda^{\omega^{key}}}$ which is not dependent on the outcome $\omega$.
It turns out that such a scheme still enables us to produce asymptotically optimal
solutions to the staffing problem. 

So, we revise the extension of model \eqref{eq:ConstraintModel} as follows and denote the model $F_{\Lambda}$:
\begin{equation}
\min_{\beta \ge 0, \omega^{key} \in \Omega} c(\beta,\omega^{key}) \\ \qquad
\mbox{s.t. } \sum_{\omega \in \Omega} p^{\omega} \mathbb{P}\left\{ W(\beta,\Lambda)> 0 \mid\Lambda=\Lambda^{\omega}\right\}\leq \epsilon.
\label{eq:Random_Single_Chance2_0_revised}
\end{equation}

To facilitate asymptotic analysis, we need to properly define how $\Lambda$ grows large.
Let the initial value of the arrival rate in all scenarios be $\Lambda_{0}=(\Lambda_{0}^{\omega_{1}},\ldots,\Lambda_{0}^{\omega_{\mid\Omega\mid}})$, where, without loss of generality, we assume the components of $\Lambda_{0}$ satisfy $\Lambda_{0}^{\omega_{1}} < \Lambda_{0}^{\omega_{2}} < \cdots< \Lambda_{0}^{\omega_{\mid\Omega\mid}}$. Next, assume that the arrival rate in the
 $m^{th}$, $m \in \mathbb{N}$,  system is $\Lambda_{m}=m\Lambda_{0}$. 
 We then let $m\rightarrow\infty$. Define now \eqref{eq:Random_Single_Chance2_0_revised} with $\Lambda_m$ as $F_{\Lambda_m}$:
\begin{equation}
\small
\min_{\beta \geq 0, \omega^{key} \in \Omega}\quad c(\beta,\omega^{key})\\ \qquad
\mbox{s.t. } \sum_{\omega\in\Omega} p^{\omega} \mathbb{P}\left\{ W(\beta,\Lambda_m)> 0 \mid\Lambda_m=\Lambda_m^{\omega}\right\}\leq \epsilon.
\label{eq:Random_Single_Chance2}
\end{equation}
The objective function $c(\beta,\omega)$ is assumed to have the following property: $c(\cdot,\omega)$ is strictly increasing and continuous for all $\omega \in \Omega$.

As implied by Theorem \ref{HWtheorem}, when $\lambda$ is deterministic the probability of waiting has a non-degenerate limit if and only if the number of servers, $n$, increases in such a way that $n=\lambda+\beta\sqrt{\lambda}$ for some $\beta>0$. In model \eqref{eq:Random_Single_Chance2}, the number of servers $n$, or equivalently $(\beta,\omega^{key})$, is chosen before we see the realization of the arrival rate. For a given staffing level $(\beta,\omega^{key})$, scenario $\omega^{key}$ is the only scenario for which the limiting probability of waiting is strictly between 0 and 1. In other scenarios, the system is either over- or under-loaded for the chosen staffing level as the arrival rate grows large. We thus obtain the following corollary of Theorem \ref{HWtheorem}.

\begin{corollary}
\label{HWcorollary}
For a given staffing level specified by $\beta > 0$ and $\omega^{key}$, we have
\[
\lim_{m \rightarrow \infty} \bar{\alpha}(\Lambda_m^{\omega^{key}} + \beta \sqrt{\Lambda_m^{\omega^{key}}},\Lambda_m^{\omega^{key}}) \in (0,1).
\]
For all $\omega \in \Omega$ such that $\omega \neq \omega^{key}$ and $\Lambda_0^\omega > \Lambda_0^{\omega^{key}}$, we have
\[
\lim_{m \rightarrow \infty} \bar{\alpha}(\Lambda_m^{\omega^{key}} + \beta \sqrt{\Lambda_m^{\omega^{key}}},\Lambda_m^{\omega})=1;
\]
for all $\omega \in \Omega$ such that $\omega \neq \omega^{key}$ and $\Lambda_0^\omega < \Lambda_0^{\omega^{key}}$, we have
\[
\lim_{m \rightarrow \infty} \bar{\alpha}(\Lambda_m^{\omega^{key}} + \beta \sqrt{\Lambda_m^{\omega^{key}}},\Lambda_m^{\omega})=0.
\]
%
\end{corollary}

Consider the constraint of model \eqref{eq:Random_Single_Chance2}, for a specific decision $\beta > 0$ and $\omega_i=\omega^{key}$. Then, we have
\begin{equation}
\label{xxx}
\lim_{m\rightarrow\infty}\bar{\alpha}(\Lambda_m^{\omega_{i}} + \beta \sqrt{\Lambda_m^{\omega_{i}}},\Lambda_m^{\omega_i})\in(0,1).
\end{equation}
Suppose we approximate the QoS constraint in \eqref{eq:Random_Single_Chance2} by
 \begin{equation}
\label{yyy}
 \quad \sum_{k=i+1}^{|\Omega|} p^{\omega_k} +p^{\omega_{i}}\mathbb{P}\left\{W(\beta,\Lambda_m)> 0 \mid\Lambda_m=\Lambda_m^{\omega_{i}}\right\}\leq \epsilon.
 \end{equation}
This approximation replaces $\mathbb{P}\left\{ W(\beta,\Lambda_m)> 0 \mid\Lambda_m=\Lambda_m^{\omega}\right\}$ by unity for $\omega = \omega_{i+1},\ldots, \omega_{|\Omega|}$, and by zero for $\omega=\omega_1,\ldots,\omega_{i-1}$. In view of Corollary \ref{HWcorollary} and equation \eqref{xxx}, this approximation becomes increasingly precise as $m$ grows large.

Equation \eqref{yyy} and the structure of $c(\beta,\omega_{i})$ suggest that for sufficiently large $m$ we should select the key scenario by finding the scenario $\omega_{i}$ such that $\sum_{k=i}^{|\Omega|}p^{\omega_{k}}\geq\epsilon$ and $\sum_{k=i+1}^{|\Omega|}p^{\omega_{k}}<\epsilon$, for $i\in\{1,2,\ldots ,|\Omega|-1\}$; if $p^{\omega_{|\Omega|}}\geq \epsilon$, then we select $\omega_{|\Omega|}$ as the key scenario. In our work, we do not consider the trivial situations where $\epsilon=0$ or $\epsilon=1$. Thus, this mechanism for selecting the key scenario yields a unique $\omega_i$. Given the key scenario $\omega_{i}$, we form a first approximation to model \eqref{eq:Random_Single_Chance2} as:
\begin{equation}
\min_{\beta \geq 0}\quad c(\beta,\omega_{i})\\ \qquad
\mbox{s.t. } p^{\omega_{i}} \mathbb{P}\left\{ W(\beta,\Lambda_m^{\omega_{i}})> 0 \right\}\leq \left(\epsilon-\sum_{k=i+1}^{|\Omega|}p^{\omega_{k}}\right).
\label{eq:Random_Single_Chance3}
\end{equation}

The term $\mathbb{P}\left\{ W\left(\beta,\Lambda_m^{\omega_{i}}\right)> 0 \right\}$ in model \eqref{eq:Random_Single_Chance3} is calculated by the Erlang-C formula, $\bar{\alpha}(\Lambda_m^{\omega_{i}} + \beta \sqrt{\Lambda_m^{\omega_{i}}},\Lambda_m^{\omega_i})$. We can use the upper bound $UB(\beta,\Lambda_m^{\omega_i})$ to approximate $\mathbb{P}\left\{ W(\beta,\Lambda_m^{\omega_{i}})> 0 \right\}$ and build our approximating model with $\omega_{i}$ which we denote $G_{\Lambda_m}$:
\begin{equation}
\min_{\beta \geq 0 }\quad c(\beta, \omega_{i})\\ \qquad
\mbox{s.t. } p^{\omega_{i}} UB\left(\beta,\Lambda_m^{\omega_{i}}\right)\leq \left(\epsilon-\sum_{k=i+1}^{|\Omega|}p^{\omega_{k}}\right).
\label{eq:Random_Single_Chance4}
\end{equation}
%

Our next goal is to extend Theorem \ref{Det_Single_optimalityResult} to the doubly 
stochastic case considered in this section. First we need Lemma \ref{Det_Single_optimalityResult_extension}, which is
proved in the Appendix. That lemma relies on a classic result from real analysis, stated below for completeness. 

\begin{theorem}
\emph{(Buchanan and Hildebrandt \cite{Buchanan_1908})}
\label{buchanan}
If a sequence $f_{n}(x)$ of monotonic functions converges to a continuous function $f(x)$ in $[a,b]$ then this convergence is uniform.
\end{theorem}

In the lemma below, we extend to the models in Section \ref{chapter-deterministic-arrival-rates}
to include constraints whose right-hand sides are also functions of $\lambda$. 
\begin{lemma}
\label{Det_Single_optimalityResult_extension}
Let $\lambda>0.$ We extend model \eqref{eq:1D_original_2} in Section~\ref{chapter-deterministic-arrival-rates} to
\begin{equation}
\min_{\beta \geq 0}\quad c(\beta)\\ \qquad
\mbox{s.t. } \tilde{\alpha}(\beta,\lambda)\leq \epsilon_{\lambda},
\end{equation}
and denote its optimal solution by $\beta_{\lambda}^{F}$. We also extend model \eqref{eq:1D_modify_3} in Section~\ref{chapter-deterministic-arrival-rates} to
\begin{equation}
\min_{\beta \geq 0}\quad c(\beta)\\ \qquad
\mbox{s.t. } UB(\beta,\lambda)\leq \epsilon_{\lambda},
\end{equation}
and denote its optimal solution by $\beta_{\lambda}^{G}$. Here the right-hand side $\epsilon_{\lambda}$ satisfies $\lim_{\lambda\rightarrow\infty}\epsilon_{\lambda}=\epsilon>0$. Then $\beta_{\lambda}^{G}\geq\beta_{\lambda}^{F}$, $\forall \lambda>0$, and there exists a finite $\beta^{*}$ such that $$\lim_{\lambda\rightarrow \infty}\beta_{\lambda}^{G}=\lim_{\lambda\rightarrow \infty}\beta_{\lambda}^{F}=\beta^{*}.$$
\end{lemma}

In the following result, we use Lemma \ref{Det_Single_optimalityResult_extension} to infer that the gap between the optimal solution for model \eqref{eq:Random_Single_Chance2} and the optimal solution for model \eqref{eq:Random_Single_Chance4} goes to 0 as the system size increases, i.e., the approximating solutions
are also asymptotically optimal in the doubly stochastic model. 

First, we introduce some notation. Let $(\omega_m^F,\beta_m^F)$ be an optimal solution to model $F_{\Lambda_m}$ as defined in \eqref{eq:Random_Single_Chance2}. For any $m$, there are
multiple such solution pairs, since the same staffing level can be achieved by different combinations
of $\omega_m$ and $\beta_m$. However, once $\omega_m$ is chosen, there exists but one
optimal $\beta_m$. For a fixed $m$, let $\mathbf{W}_m$ be the set of all optimal
$\omega_m^F$.

\begin{theorem}
\label{Random_Single_optimalityResult}
Let $\beta_m^G$ be an optimal solution to model $G_{\Lambda_m}$ as defined in \eqref{eq:Random_Single_Chance4}. Assume $\epsilon \in (0,1)$ is such that there exists an $i$ with $\sum_{k=i}^{|\Omega|} p^{\omega_k} > \epsilon$ and $\sum_{k=i+1}^{|\Omega|} p^{\omega_k} < \epsilon$. 
Then, there exists an $\bar{m}$ such that for all $m \ge \bar{m}$ we have $\omega_i \in \mathbf{W}_m$. And, there 
exists a $\beta^* > 0$ such that
\[
\lim_{m\rightarrow \infty}\beta_{m}^{G}=\lim_{m\rightarrow \infty}\beta_{m}^{F}=\beta^{*},
\]
where the $\beta_{m}^{F}$ are the optimal staffing factors for \eqref{eq:Random_Single_Chance2}
with $\omega_m^F= \omega_i$ for all $m$.
\end{theorem}
\begin{proof}
 In what follows we use $\omega < \omega'$ to mean $\Lambda_0^{\omega} < \Lambda_0^{\omega'}$. Since $|\Omega|$ is finite, from Corollary \ref{HWcorollary} we have $$\lim_{m\rightarrow\infty}\max_{\omega, \omega'\in\Omega, \omega \neq \omega'}\min\left\{\bar{\alpha}(\Lambda_m^{\omega'}+\beta\sqrt{\Lambda_m^{\omega'}},\omega), 1-\bar{\alpha}(\Lambda_m^{\omega'}+\beta\sqrt{\Lambda_m^{\omega'}},\omega)\right\}=0,$$ 
for each fixed $\beta >0$. 
 
 Thus, given $\Delta>0$ and a fixed $\beta >0$, there exists an $\bar{m}$, such that for all $m\geq\bar{m}$
\begin{equation*}
\label{delta}
\max_{\omega\in\Omega, \omega<\omega_m^F}\left\{\bar{\alpha}(\Lambda_m^{\omega_m^F}+\beta\sqrt{\Lambda_m^{\omega_m^F}},\omega)\right\}\leq\Delta,
\end{equation*}
and
\begin{equation*}
\label{delta2}
\max_{\omega\in\Omega, \omega>\omega_m^F}\left\{1-\bar{\alpha}(\Lambda_m^{\omega_m^F}+\beta\sqrt{\Lambda_m^{\omega_m^F}},\omega)\right\}\leq\Delta.
\end{equation*}
Hence, for all $m\geq\bar{m}$, the left-hand side of the constraint in \eqref{eq:Random_Single_Chance2} is bounded above by
\begin{equation}
\small
\label{reduced_constraint_UB}
\left(\sum_{\omega\in\Omega,\omega<\omega_m^F}p^{\omega}\right)\Delta+\sum_{\omega\in\Omega,\omega>\omega_m^F}p^{\omega}+p^{\omega_m^F}\bar{\alpha}(\Lambda_m^{\omega_m^F}+\beta\sqrt{\Lambda_m^{\omega_m^F}},\omega_m^F)
\end{equation}
and bounded below by
\begin{equation}
\small
\label{reduced_constraint_LB}
\left(\sum_{\omega\in\Omega,\omega>\omega_m^F}p^{\omega}\right)(1-\Delta)+p^{\omega_m^F}\bar{\alpha}(\Lambda_m^{\omega_m^F}+\beta\sqrt{\Lambda_m^{\omega_m^F}},\omega_m^F).
\end{equation}
Then for $\Delta$ chosen sufficiently small, we have the following conclusions. 
For any $m\geq\bar{m}$, if $\omega_m^F < \omega_i$, \eqref{reduced_constraint_LB} indicates a contradiction of feasibility for $F_{\Lambda_m}$. If $\omega_m^F > \omega_i$, \eqref{reduced_constraint_UB} indicates that model \eqref{eq:Random_Single_Chance2} is feasible for all $m\geq \bar{m}$. However, model \eqref{eq:Random_Single_Chance2} is also feasible in the $\omega_i$ case for all $m\geq \bar{m}$. 
This indicates that there exists an $\bar{m}$, such that for all $m\geq\bar{m}$ we have $\omega_i \in \mathbf{W}_m$. Then, for all $m\geq\bar{m}$, with $\omega_i$ fixed, we can re-write \eqref{eq:Random_Single_Chance2} as a deterministic model as follows:
\begin{equation}
\small
\min_{\beta \geq 0}\quad c(\beta,\omega_{i})\\ \qquad
\mbox{s.t. } p^{\omega_{i}} \mathbb{P}\left\{ W(\beta,\Lambda_m^{\omega_{i}})> 0 \right\}\leq \left(\epsilon-\sum_{k=i+1}^{|\Omega|}p^{\omega_{k}}\right)+\Delta_m.
\label{reduced_constraint_2}
\end{equation}
Here $\lim_{m\rightarrow\infty}\Delta_m=0$. Applying Lemma \ref{Det_Single_optimalityResult_extension}, we have $$\lim_{m\rightarrow \infty}\beta_{m}^{G}=\lim_{m\rightarrow \infty}\beta_{m}^{F}=\beta^{*}.$$
\end{proof}
Due to the continuity assumption on the cost function, the result above also implies
that the objective function values converge. 

\subsection{Multi-station Systems}
We now consider a multi-station system, again assuming that there are $L$ queues, whose dynamics are conditionally independent. We define $\Lambda=(\Lambda_{1},\ldots,\Lambda_{L})$ to be the random arrival rate vector. Let $\Lambda^{\omega}$ be a specific realization, where $\omega=(\omega_{1},\ldots,\omega_{L})$ is a sample point from the finite sample space $\Omega=\Omega_{1}\times\cdots\times\Omega_{L}$. Let $p^{\omega}$ be the probability assigned to scenario $\omega$.
We consider the model below:
\begin{equation}
\begin{split}
&\min_{\beta \geq 0, \omega^{key}\in \Omega}\quad \sum_{i=1}^{L}c_{i}(\beta_{i}, \omega_{i}^{key})
\qquad \mbox{s.t. } \sum_{\omega\in\Omega} p^{\omega} \mathbb{P}\left\{ \left.\bigcup_{i=1}^{L}\left\{W_{i}(\beta_{i},\Lambda_{i})> 0\right\}  \right| \Lambda=\Lambda^{\omega}\right\}\leq \epsilon,
\label{eq:Random_Chance1}
\end{split}
\end{equation}
which is equivalent to the following model:
\begin{equation}
\begin{split}
&\min_{\beta \geq 0, \omega^{key}\in \Omega}\quad \sum_{i=1}^{L}c_{i}(\beta_{i}, \omega_{i}^{key})\\ \qquad
&\mbox{s.t. } \sum_{\omega\in\Omega} p^{\omega}\prod_{i=1}^{L}\left.\mathbb{P}\left\{ W_{i}(\beta_{i},\Lambda_{i})= 0\right| \Lambda_i=\Lambda_i^{\omega_i} \right\} \geq 1-\epsilon,
\label{eq:Random_Chance3}
\end{split}
\end{equation}
where
$$\mathbb{P}\left\{ W_{i}(\beta_{i},\Lambda_{i})= 0|\Lambda_{i}=\Lambda_{i}^{\omega_i}\right\}=1-\bar{\alpha}(\Lambda_{i}^{\omega_i^{key}}+\beta_{i}\sqrt{\Lambda_{i}^{\omega_i^{key}}},\Lambda_{i}^{\omega_i}).$$
Again, the cost function $c_i(\cdot, \cdot)$ is assumed to have the following property: $c_i(\cdot,\omega)$ is strictly increasing and continuous for all $\omega \in \Omega$ for $i=1, \ldots, L$. 

Facing this random arrival-rate model, we may think that instead of solving the joint model \eqref{eq:Random_Chance3}, it would be easier to solve several single-station models. That is, we can treat the $L$ queues individually, and 
obtain the optimal staffing policy for each queue separately. For example, instead of solving model \eqref{eq:Random_Chance3}, we consider solving the following set of individual models, indexed by
$i$:
%
\begin{equation}
\begin{split}
\label{eq:Chance3_1}
&\min_{\beta_{i} \geq 0, \omega_{i}^{key}\in \Omega_{i}}\quad c_{i}(\beta_{i}, \omega_{i}^{key})\\ \qquad
&\mbox{s.t. }\sum_{\omega_{i}\in\Omega_{i}}p_{i}^{\omega_{i}}\mathbb{P}\left\{ W_{i}(\beta_{i},\Lambda_{i})= 0 \mid \Lambda_{i}=\Lambda_{i}^{\omega_i} \right\}\\ \quad
&\qquad\geq \sqrt[ L]{1-\epsilon}, \quad i=1,2,\ldots,L,
\end{split}
\end{equation}
where $p_{i}^{\omega_{i}}$ is the marginal probability of scenario $\omega_{i}\in\Omega_{i}$.

When decomposing the joint model \eqref{eq:Random_Chance3} into the models in \eqref{eq:Chance3_1}, it is difficult to decide on the right-hand side of the constraint in each individual problem. In \eqref{eq:Chance3_1}, we set the right-hand side of each to be $\sqrt[L]{1-\epsilon}$ with the notion that the service level should be the same across the individual stations if there is no managerial reason to favor one station over another. 
Although decomposing the problem in this manner improves tractability, the resulting solution may be
undesirable, as shown in the following example.

\begin{example}
\label{mul_random_EX}
Let $L=2$ and consider the $M/M/n$ system as shown in Figure \ref{ExPic}. Suppose each queue has random arrival rate $\Lambda_{i}$, $i=1,2$. Assume there are two scenarios for the arrival rate of queue 1 (high and low) and there are three scenarios for the arrival rate of queue 2 (high, medium and low). The joint probability distribution is given in Table \ref{table_1}. The realizations of $\Lambda$ for each queue under each scenario are given in Table \ref{table_2}. We assume a linear cost with cost coefficients $c_{1}=\$5$, $c_{2}=\$3$ and a service level threshold value of $\epsilon=0.05$.

\begin{figure}[htp]
\begin{center}
\includegraphics[scale=0.4]{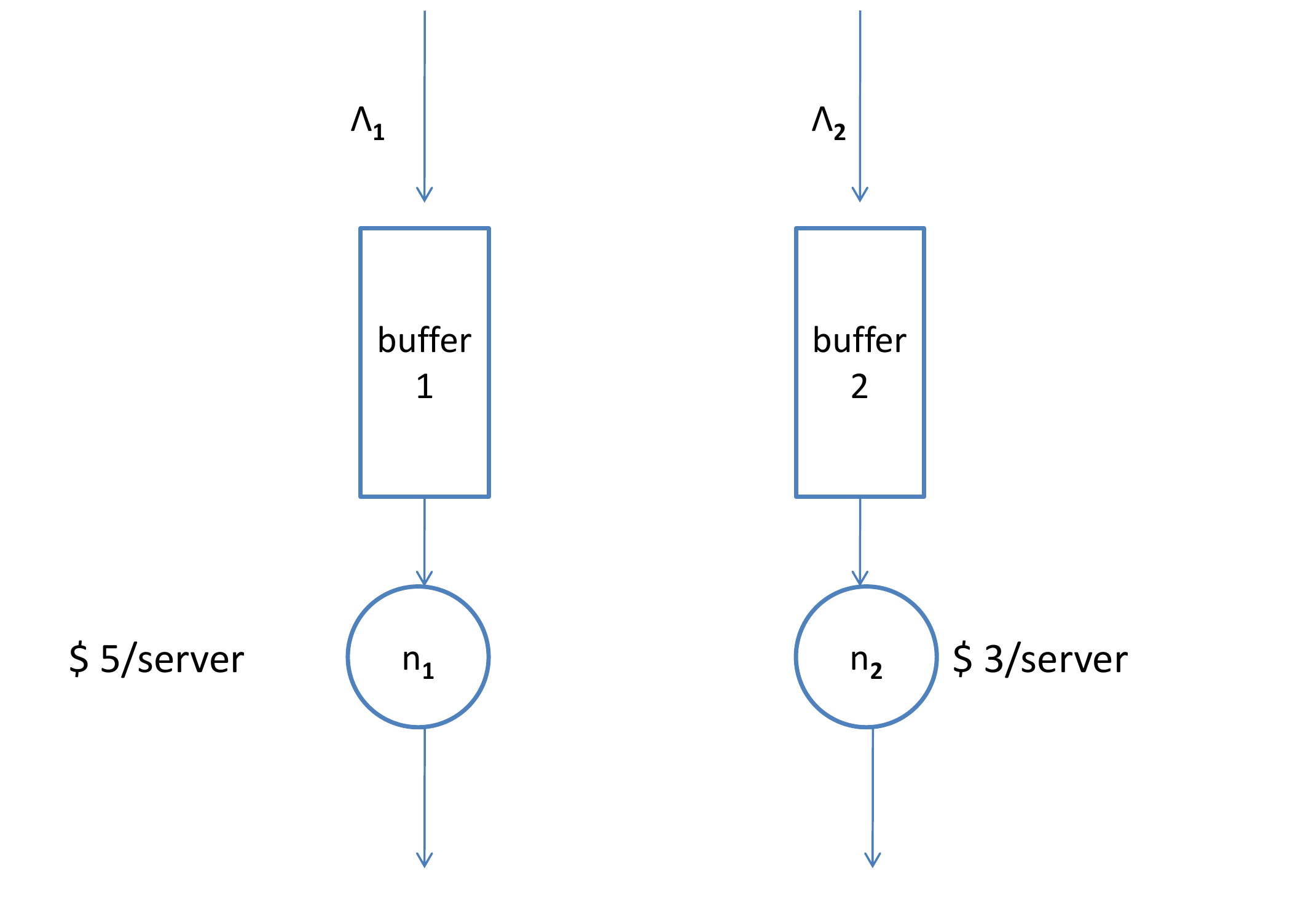}
\caption{Example \ref{mul_random_EX} System}
\label{ExPic}
\end{center}
\end{figure}

\begin{table}[htp]
 \begin{center}
 \caption{Joint Probability for Example \ref{mul_random_EX}}
  \begin{tabular}{ | l | l | l | l |}
   \hline
     $p^{(\omega_1,\omega_2)}$& $\omega_{2}=high$ & $\omega_{2}=medium$ & $\omega_{2}=low$ \\ \hline
    $\omega_{1}=high$ & 0.03& 0.21 & 0.1 \\ \hline
    $\omega_{1}=low$ & 0.01 & 0.17 &0.48 \\ \hline
  \end{tabular}
  \label{table_1}
  \end{center}
  \end{table}

\begin{table}[htp]
 \begin{center}
 \caption{Arrival Rates for Example \ref{mul_random_EX}}
  \begin{tabular}{ | l | l | l | l |}
   \hline
     & high & medium & low \\ \hline
    $\Lambda_1$ for queue 1& 450& NA & 350 \\ \hline
    $\Lambda_2$ for queue 2& 300 & 200 &100 \\ \hline
  \end{tabular}
  \label{table_2}
  \end{center}
  \end{table}

\end{example}

For the given parameters we solve the exact (non-asymptotic) versions of the corresponding models
as given in \eqref{eq:Random_Chance3} and \eqref{eq:Chance3_1}. From the solutions in Table \ref{table_3} we can see that while both achieve the same service level, the cost of the staffing policy from the decoupled models \eqref{eq:Chance3_1} is about $5\%$ more than that of the joint model \eqref{eq:Random_Chance3}.

\begin{table}[htp]
 \begin{center}
 \caption{Solution Comparison}
   \begin{tabular}{ | l | l | l | l |}
    \hline
      & exact model \eqref{eq:Random_Chance3}& decoupled models \eqref{eq:Chance3_1} \\ \hline
     $n$ &$(n_{1}^{*}=496, n_{2}^{*}=235)$& $(\overline{n}_{1}=484, \overline{n}_{2}=306)$\\ \hline
     cost $(c_{1}n_{1}+c_{2}n_{2})$ & 3185 & 3338\\ \hline
     $\mathbb{E}_{\Lambda}\left[\mathbb{P}\left\{\bigcup_{i=1}^{2}wait_{i}(n_{i},\Lambda_{i})> 0\right\} \right]$& 0.05 & 0.05\\ \hline
   \end{tabular}
   \label{table_3}
   \end{center}
   \end{table}
 \qed

We now discuss how to formulate an asymptotics for joint model \eqref{eq:Random_Chance3}. Again, we consider a sequence of queueing systems with increasing arrival rates. The asymptotics of the arrival
rates are similar to the scheme we considered in the single-station system. 
We assume there is a base value of the arrival rate in all scenarios. Let this base rate be $\Lambda^{0}=({\Lambda^{0}}^{\omega_{1}},\ldots,{\Lambda^{0}}^{\omega_{|\Omega|}})$, where each ${\Lambda^{0}}^{\omega_{k}}, k=1,\ldots,|\Omega|$, is an $L$-vector, since it represents the base arrival rate the for $L$-station system in scenario $\omega_{k}$. 
For $m \in \mathbb{N}$ let the arrival rate in the $m^{th}$ system be $\Lambda_{m}=m\Lambda_{0}$
and again let $m\rightarrow\infty$.

As in the single-station system, the manager must pick the square-root safety factor $\beta_{i}, i=1,\ldots,L$  for each queue, before the realization of $\Lambda_{i}, i=1,\ldots,L$. Thus once the staffing factors are fixed, then in at most one scenario, $\omega^{key}=(\omega_{1}^{key},\ldots,\omega_{L}^{key})$, the probability of waiting in each queue converges to a value strictly between 0 and 1 for large arrival rates. That is we can find a key scenario $\omega^{key}$, such that $\forall\omega\neq\omega^{key}$ and $\forall k\in\{1,\ldots,L\}$ if $\Lambda_{k}^{\omega_{k}}>\Lambda_{k}^{\omega_{k}^{key}}$, the limiting probability of not waiting for service in queue $k$ under this scenario is 0; and, if $\Lambda_{k}^{\omega_{k}}<\Lambda_{k}^{\omega_{k}^{key}}$, the limiting probability of not waiting for service in queue $k$ under this scenario is 1. Thus, if a key scenario can be identified, the random parameter model reduces to a deterministic model. However, unlike the single-station system, it is not easy to identify a key scenario. Just as a multivariate distribution does not have a unique quantile, in this problem, the key scenario need not be unique, since the QoS ``risk'' can
be spread in a number of ways. If one knows in advance how to allocate the probability embodied
in $\epsilon$ to each station, then it is relatively easy to identify the key scenario and for any such
allocation, the existence of a key scenario for each station is guaranteed using the previous single-station arguments. In theory, an integer programming model can be created to find the allocation of
$\epsilon$ which minimizes the staffing costs.

We are now ready to present the stochastic arrival rate versions of the results in Section
\ref{sec:mss}. As in the deterministic rate case, we first formulate the ``dualized'' version
of (\ref{eq:Random_Chance3}), a model which we denote by $F_{\Lambda}$: 
\begin{equation}
\min_{\beta \geq 0, \omega^{key}\in \Omega}\quad \sum_{i=1}^{L}c_{i}(\beta_{i}, \omega_i^{key})+\delta\left(1-\sum_{\omega\in\Omega} p^{\omega}\prod_{i=1}^{L}\left.\mathbb{P}\left\{ W_{i}(\beta_{i},\Lambda_{i})= 0\right| \Lambda_i=\Lambda_i^{\omega_i} \right\}\right).\\ \qquad
\label{mss:flambda}
\end{equation}
As before, the corresponding approximate model is denoted $G_{\Lambda}$:
\begin{equation}
\min_{\beta \geq 0, \omega^{key}\in \Omega}\quad \sum_{i=1}^{L}c_{i}(\beta_{i},\omega_i^{key})+\delta\left(1- \sum_{\omega\in\Omega} p^{\omega}\left\{\prod_{i=1}^{L}\left(1-UB(\beta_{i},\Lambda_{i}^{\omega})\right)\right\}\right).\\ \qquad
\label{mss:glambda}
\end{equation}

With these formulations, we now have the following extension of Lemma \ref{Uniform_Converge_Multi_Term}:

\begin{lemma}
\label{mss:uc}
Let $f_{m}(\cdot, \cdot)$ denote the objective function of model $F_{\lambda}$ as defined in \eqref{mss:flambda}, with arrival rate $\Lambda^m$. And, let $g_{m}(\cdot, \cdot)$ denote the objective function of model $G_{\lambda}$ as defined in \eqref{mss:glambda}, with arrival rate $\Lambda^m$. Then, $$\lim_{m\rightarrow\infty}\sup_{\beta\geq0, \omega^{key} \in \Omega}\left(g_{m}(\beta, \omega^{key} )-f_{m}(\beta, \omega^{key})\right)=0.$$
\end{lemma}
Since we assume that the sample space $\Omega$ is finite, the lemma follows from minor modifications
to the Lemma \ref{Uniform_Converge_Multi_Term} proof. We now present our final asymptotic optimality
result, which relates the models presented in \eqref{mss:flambda} and \eqref{mss:glambda}.

\begin{theorem}
\label{mss:main}
Let $f_{m}(\cdot, \cdot)$ denote the objective function and let $(\beta_{m}^{F}, \omega^{F}_m)$ denote an optimal solution of model $F_{\Lambda}$ as defined in \eqref{mss:flambda}, with arrival rate $\Lambda^m$. Let $g_{m}(\cdot, \cdot)$ denote the objective function, and let $(\beta_{m}^{G}, \omega^{G}_m)$ denote an optimal solution of model $G_\Lambda$ as defined in \eqref{mss:glambda}, with arrival rate $\Lambda^m$. Then $$\lim_{m\rightarrow\infty}\left(f_{m}(\beta_{m}^{G}, \omega^{G}_m)-f_{m}(\beta_{m}^{F},\omega^{F}_m)\right)=0.$$
\end{theorem}
The theorem follows by completely analogous arguments used to prove Theorem \ref{multi_objvalue_converge}, applying Lemma \ref{mss:uc}
instead of Lemma \ref{Uniform_Converge_Multi_Term}.

Finally, we use the data in Example \ref{mul_random_EX} to illustrate how the key scenario idea can
be used to solve a multi-station problem. It is obvious that the key scenario in this example is $(\omega_1, \omega_2)=(high, medium)$. Otherwise the value of the left-hand side of the constraint in model \eqref{eq:Random_Chance3} cannot exceed $1-\epsilon$. Also, with the key scenario being $(high, medium)$, we can select $(\beta_{1}, \beta_{2})$ to satisfy the constraint. Thus it is not necessary to consider scenario $(high, high)$, which is more costly. After finding the key scenario for Example \ref{mul_random_EX}, we can write model \eqref{eq:Random_Chance3} for the asymptotic version of the original random rates problem as:
\begin{equation}
\begin{split}
\label{eq:mul_random_EX}
&\min_{\beta \geq 0}\quad 5\beta_{1}+3\beta_{2}\\ \qquad
&\mbox{s.t. } 0.21\left(1-\tilde{\alpha}(\beta_{1},450)\right)(1-\tilde{\alpha}(\beta_{2},200))\\
&+0.1(1-\tilde{\alpha}(\beta_{1},450))+0.17(1-\tilde{\alpha}(\beta_{2},200))\\ \quad
&+0.48 \geq (1-0.05).
\end{split}
\end{equation}
%
Solving model (\ref{eq:mul_random_EX}) we obtain the optimal solution $(\beta_1^{*},\beta_2^{*})=(2.15,2.48)$, which gives $$(n_1^{*},n_2^{*})=\left(450+2.15\cdot\sqrt{450},\quad 200+2.48\cdot\sqrt{200}\right)\approx (496, 235).$$

We now solve the problem using the individual models as presented in \eqref{eq:Chance3_1}. The individual models for station 1 and station 2 are:
\begin{equation}
\begin{split}
&\min_{\beta_{1} \geq 0}\quad 5\beta_{1}\\ \qquad
&\mbox{s.t. } 0.34(1-\tilde{\alpha}(\beta_{1},450)) +0.66(1-\tilde{\alpha}(\beta_{1},350))\geq \sqrt{1-0.05},
\label{eq:mul_random_EX_1}
\end{split}
\end{equation}
and
\begin{equation}
\begin{split}
&\min_{\beta_{2} \geq 0}\quad 3\beta_{2}\\ \qquad
&\mbox{s.t. } 0.04(1-\tilde{\alpha}(\beta_{2},300)) +0.38(1-\tilde{\alpha}(\beta_{2},200))\\ \qquad
&+0.58(1-\tilde{\alpha}(\beta_{2},150))\geq \sqrt{1-0.05}.
\label{eq:mul_random_EX_2}
\end{split}
\end{equation}
The key scenarios for queue 1 and queue 2 are both $high$. The individual models above are equivalent to:
\begin{equation}
\min_{\beta_{1} \geq 0}\quad 5\beta_{1}\\ \qquad
\mbox{s.t. } 0.34(1-\tilde{\alpha}(\beta_{1},450)) +0.66 \geq \sqrt{1-0.05},
\label{eq:mul_random_EX_1_2}
\end{equation}
and
\begin{equation}
\min_{\beta_{2} \geq 0}\quad 3\beta_{2}\\ \qquad
\mbox{s.t. } 0.04(1-\tilde{\alpha}(\beta_{2},300)) +0.38+0.58 \geq \sqrt{1-0.05}.
\label{eq:mul_random_EX_2_2}
\end{equation}
The optimal  to solutions to \eqref{eq:mul_random_EX_1_2} and \eqref{eq:mul_random_EX_2_2} are $\overline{\beta}_1=1.6$, $\overline{\beta}_2=0.36$. This gives $\overline{n}_1=450+1.6\cdot\sqrt{450}\approx 484$, $\overline{n}_2=300+0.36\cdot\sqrt{300}\approx 306$.
As noted in Table \ref{table_3}, the cost of the solution obtained via decoupling the stations is about 5\% higher 
than the cost obtained using the joint model. 

\textbf{Acknowledgements.} This research was supported by National Science Foundation grant  CMMI-0800676. Any opinions, findings and conclusions or recommendations expressed in this material are those of the author(s) and do not necessarily reflect the views of the National Science Foundation (NSF).



\begin{thebibliography}{10}

\bibitem{bhz06}
A.~Bassamboo, J.~M. Harrison, and A.~Zeevi.
\newblock Design and control of a large call center: {A}symptotic analysis of
  an {LP}-based method.
\newblock {\em Operations Research}, 54:419--435, 2006.

\bibitem{brz10}
A.~Bassamboo, R.~S. Randhawa, and A.~Zeevi.
\newblock Capacity sizing under parameter uncertainty: {S}afety staffing
  principles revisited.
\newblock {\em Management Science}, 56:1668--1686, 2010.

\bibitem{bassamboo_07}
A.~Bassamboo and A.~Zeevi.
\newblock On a data-driven method for staffing large call centers.
\newblock {\em Operations Research}, 57(3):714--726, 2009.

\bibitem{borst_04}
S.~Borst, A.~Mandelbaum, and M.~I. Reiman.
\newblock {Dimensioning large call centers}.
\newblock {\em Operations Research}, 52(1):17--34, 2004.

\bibitem{Buchanan_1908}
H.~E. Buchanan and T.~H. Hildebrandt.
\newblock Note on the convergence of a sequence of functions of a certain type.
\newblock {\em Annals of Mathematics}, 9(2):123--126, 1908.

\bibitem{gurvich_10}
I.~Gurvich, J.~Luedtke, and T.~Tezcan.
\newblock Staffing call centers with uncertain demand forecasts: A
  chance-constrained optimization approach.
\newblock {\em Management Science}, 56(7):1093--1115, 2010.

\bibitem{Halfin_81}
S.~Halfin and W.~Whitt.
\newblock Heavy-traffic limits for queues with many exponential servers.
\newblock {\em Operations Research}, 29(3):567--588, 1981.

\bibitem{harrison_05}
J.~M. Harrison and A.~Zeevi.
\newblock A method for staffing large call centers based on stochastic fluid
  models.
\newblock {\em Manufacturing and Service Operations Management}, 7(1):20--36,
  2005.

\bibitem{jagers_vandoorn_86}
A.~A. Jagers and E.~A. Van~Doorn.
\newblock {On the continued Erlang loss function}.
\newblock {\em Operations Research Letters}, 5(1):43--46, 1986.

\bibitem{janssen_08}
A.~J. E.~M. Janssen, J.~S. H.~Van Leeuwaarden, and B.~Zwart.
\newblock Refining square root safety staffing by expanding {E}rlang {C}.
\newblock {\em Operations Research}, 59(6):1512--1522, 2011.

\bibitem{kaw13}
Y.~L. Kocaga, M.~Armony, and A.~R. Ward.
\newblock Staffing and admission control in an {M/M/N+N} queue with an
  uncertain arrival rate.
\newblock 2013.
\newblock Working paper.

\bibitem{Ralphs_06}
T.~K. Ralphs, M.~J. Saltzman, and M.~M. Wiecek.
\newblock An improved algorithm for solving biobjective integer programs.
\newblock {\em Annals of Operations Research}, 147:43--70, 2006.

\bibitem{whi06}
W.~Whitt.
\newblock Staffing a call center with uncertain arrival rate and absenteeism.
\newblock {\em The Annals of Applied Probability}, 14(1):88--102, 2006.

\end{thebibliography}

\section*{Appendices}
\label{chapter_appendices}

\begin{proof}[Proof of Lemma \ref{UBDecrease}]
\label{UBDecreaseProof}
To show $UB(\beta,\lambda)$ is strictly decreasing in $\lambda$, it suffices to show the denominator of \eqref{eq:UB}, $\rho+\gamma\left(\frac{\Phi(a)}{\phi(a)}+\frac{2}{3\sqrt{n}}\right)$, is strictly increasing in $\lambda$ for any $\beta>0$. First note that $$\rho+\gamma \left(\frac{2}{3\sqrt{n}}\right)=\frac{3\lambda+2\beta\sqrt{\lambda}}{3\left(\lambda+\beta\sqrt{\lambda}\right)}$$ is strictly increasing in $\lambda$ for any $\beta>0$, since $$\frac{\partial\left[\frac{3\lambda+2\beta\sqrt{\lambda}}{3(\lambda+\beta\sqrt{\lambda})}\right]}{\partial \lambda}=\frac{\beta\sqrt{\lambda}}{6(\lambda+\beta\sqrt{\lambda})^2}>0 \quad \forall\beta, \lambda>0.$$ Thus, we only need to prove that $\frac{\gamma\Phi(a)}{\phi(a)}$ is non-decreasing in $\lambda$ for any $\beta, \lambda>0$. Let $$f(\beta,\lambda)=\gamma\frac{\Phi(a)}{\phi(a)}=\beta \sqrt{\frac{\lambda}{\lambda+\beta\sqrt{\lambda}}} \frac{\Phi(a)}{\phi(a)}.$$ We can show that $f(\beta,\lambda)$ is strictly increasing by verifying that $\frac{\partial f}{\partial \lambda}>0$ for any $\beta, \lambda>0$. We have
\begin{equation}
\label{partialf}
\frac{\partial f}{\partial \lambda}=\frac{\beta ^2}{4 (\lambda+\beta\sqrt{\lambda})\sqrt{\lambda+\beta\sqrt{\lambda}}}\frac{\Phi(a)}{\phi(a)}+\beta\sqrt{\frac{\lambda}{\lambda+\beta\sqrt{\lambda}}}\frac{[\phi(a)+\Phi(a)a]\frac{\partial a}{\partial\lambda}}{\phi(a)}.
\end{equation}
From \eqref{partialf}, it is obvious that $\frac{\partial f}{\partial\lambda}>0$ for any $\beta, \lambda>0$, if $\frac{\partial a}{\partial\lambda}>0$ for any $\beta, \lambda>0$. We have
\begin{equation}
\label{partialalpha}
\frac{\partial a}{\partial \lambda}=\frac{-\frac{\beta}{\sqrt{\lambda}}+\left(1+\frac{\beta}{2\sqrt{\lambda}}\right)\ln\left(1+\frac{\beta}{\sqrt{\lambda}}\right)}{\sqrt{-2\left(\beta\sqrt{\lambda}+\lambda\right)\left(1-\frac{\lambda}{\beta\sqrt{\lambda}+\lambda}+\ln\left(\frac{\lambda}{\beta\sqrt{\lambda}+\lambda}\right)\right)}}.
\end{equation}
The denominator of \eqref{partialalpha} is strictly positive, so it suffices to show $$g(\beta,\lambda)=-\frac{\beta}{\sqrt{\lambda}}+\left(1+\frac{\beta}{2\sqrt{\lambda}}\right)\ln\left(1+\frac{\beta}{\sqrt{\lambda}}\right)>0,\quad  \forall \beta, \lambda >0.$$ We have
\begin{equation*}
\label{limit}
\lim_{\lambda\rightarrow\infty}g(\beta,\lambda)=0, \quad\forall \beta>0,
\end{equation*}
and
\begin{equation*}
\label{partialg}
\frac{\partial g}{\partial\lambda}=\frac{\beta}{2\lambda\sqrt{\lambda}}\left[\frac{\frac{\beta}{\sqrt{\lambda}}}{1+\frac{\beta}{\sqrt{\lambda}}}-\ln\left(1+\frac{\beta}{\sqrt{\lambda}}\right)\right].
\end{equation*}
Since $$\frac{x}{1+x}<\ln(1+x), \forall x>0,$$ we have $$\frac{\frac{\beta}{\sqrt{\lambda}}}{1+\frac{\beta}{\sqrt{\lambda}}}-\ln\left(1+\frac{\beta}{\sqrt{\lambda}}\right)<0, \quad\forall \beta, \lambda >0.$$ Thus we have $\frac{\partial g}{\partial\lambda}<0$ and $$\lim_{\lambda\rightarrow\infty}g\left(\beta,\lambda\right)=0, \quad\forall \beta>0.$$ This proves $g(\beta,\lambda)>0, \forall \beta, \lambda>0$.
\end{proof}

\begin{proof}[Proof of Lemma \ref{1}] 
With $$g(\beta,\lambda)=\frac{\gamma}{(12n-1)}=\frac{\beta\sqrt{\lambda/\left(\lambda+\beta\sqrt{\lambda}\right)}}{12\left(\lambda+\beta\sqrt{\lambda}\right)-1},$$
we have that $g(\beta,\lambda)>0$ and is continuous in $(\beta,\lambda)$ for all $\beta>0, \lambda>10.$ (Below, we let $\lambda$ grow large. The value 10 here simply serves as a sufficiently large lower bound we use in establishing the desired result.) We have
\begin{equation*}
\begin{split}
\frac{\partial g(\beta,\lambda)}{\partial\beta}=-\frac{12\beta\sqrt{\lambda}\sqrt{\frac{\lambda}{\lambda+\beta\sqrt{\lambda}}}}{\left(-1+12\left(\lambda+\beta\sqrt{\lambda}\right)\right)^2} +
\frac{\sqrt{\frac{\lambda}{\lambda+\beta\sqrt{\lambda}}}}{-1+12\left(\lambda+\beta\sqrt{\lambda}\right)}\\-\frac{\beta\lambda^{3/2}}{2\sqrt{\frac{\lambda}{\lambda+\beta\sqrt{\lambda}}}\left(\lambda+\beta\sqrt{\lambda}\right)^{2}\left(-1+12\left(\lambda+\beta\sqrt{\lambda}\right)\right)}.
\end{split}
\end{equation*}
Evaluating
$$\frac{\partial g(\beta,\lambda)}
{\partial \beta}=0,$$
yields, after some algebra, $$\frac{\sqrt{\frac{\sqrt{\lambda}}{\beta+\sqrt{\lambda}}}\left(-12\beta^{2}\sqrt{\lambda}+\beta(-1+12\lambda)+2\sqrt{\lambda(-1+12\lambda)}\right)}{-1+12\beta\sqrt{\lambda}+12\lambda}=0,$$ or equivalently, $$-12\beta^{2}\sqrt{\lambda}+\beta(-1+12\lambda)+2\sqrt{\lambda(-1+12\lambda)}=0.$$ Then we have $$\widehat{\beta}=\frac{-1+12\lambda+\sqrt{1-120\lambda+1296\lambda^2}}{24\sqrt{\lambda}},$$ as the only positive root of this equation.
Also we have $g(0,\lambda)=0, \forall\lambda>10$, $$\lim_{\beta\rightarrow\infty}g(\beta,\lambda)=0, \quad\forall\lambda>10,$$ and $g(\widehat{\beta},\lambda)>0, \forall\lambda>10$. Thus $\widehat{\beta}$ is the global maximizer of $g\left(\beta,\lambda\right)$ for any $\lambda>10$. That is, $$g(\widehat{\beta},\lambda)=\max_{\beta>0} g\left(\beta,\lambda\right), \quad\forall \lambda>10.$$
Since $$\lim_{\lambda\rightarrow\infty}g(\widehat{\beta},\lambda)=0,$$ we have $$\lim_{\lambda\rightarrow\infty}\sup_{\beta>0}g(\beta,\lambda)=0.$$ This completes the proof. \end{proof}

\begin{proof}[Proof of Lemma \ref{3}] 
To show that $\rho\phi(a)+\gamma\Phi(a)$ is strictly increasing in $\beta$ for any sufficiently large $\lambda$, it suffices to demonstrate that $\rho\phi(a)+\gamma\Phi(a)$ is strictly increasing in $n$ for any sufficiently large $\lambda$, since $\beta$ and $n$ satisfy a linear relationship with a positive slope. Let $h(n,\lambda)=\rho\phi(a)+\gamma\Phi(a)$. We have $$\frac{\partial h(n,\lambda)}{\partial n}=\frac{-\lambda\phi(a)}{n^2}+\Phi(a)\frac{n+\lambda}{2n\sqrt{n}}+\frac{\lambda\phi(a)\ln\left(\frac{\lambda}{n}\right)}{n}-\frac{(n-\lambda)\phi(a)\ln\left(\frac{\lambda}{n}\right)}{a\sqrt{n}}.$$
First note that $$\frac{-\lambda\phi(a)}{n^2}+\Phi(a)\frac{n+\lambda}{2n\sqrt{n}}>\frac{-\lambda\phi(a)+n\Phi(a)}{n^2},$$ since $\textit{n}>\lambda\geq 1$. (Here, we take 1 as a lower bound on $\lambda$ since we establish a result for $\lambda$ that is sufficiently large.) Also since $$\frac{-\lambda\phi(a)+n\Phi(a)}{n^2}\geq \frac{-\lambda\phi(0)+n\Phi(0)}{n^2}> 0,$$ we have $$\frac{-\lambda\phi(a)}{n^2}+\Phi(a)\frac{n+\lambda}{2n\sqrt{n}}>0.$$ It remains then to show that $$\frac{\lambda\phi(a)\ln\left(\frac{\lambda}{n}\right)}{n}-\frac{(n-\lambda)\phi(a)\ln\left(\frac{\lambda}{n}\right)}{a\sqrt{n}}\geq0\quad\forall\beta>0, \lambda>0.$$
Some algebra demonstrates that this is
equivalent to $\lambda a \leq\sqrt{n}(n-\lambda)$. As shown by equation (50) in \cite{janssen_08}, $$a=\beta-\frac{1}{6}\beta^{2}\frac{1}{\sqrt{\lambda}}+O\left(1/\lambda\right).$$ So we have
$0<a<\beta$ for sufficiently large $\lambda$. Thus $$\lambda a<\lambda\beta=\sqrt{\lambda}(n-\lambda)<\sqrt{n}(n-\lambda),$$ since $\lambda<n$.
\end{proof}

\begin{proof}[Proof of Lemma \ref{2}]
All four terms in the formula on the left-hand side of inequality \eqref{211} are non-negative for all $\beta>0$, $\lambda\geq 1$. Thus it suffices to show that $h(\beta,\lambda)=\rho\phi(a)+\gamma\Phi(a)$
is uniformly bounded away from 0 for all sufficiently large $\lambda$ and $\beta>0$, since $\phi(\cdot)$ is positive and bounded. From Lemma \ref{3}, we know $h(\beta,\lambda)$ is strictly increasing in $\beta$ for all sufficiently large $\lambda$ and $h(0,\lambda)=\phi(0)>0$. Thus we have that $h(\beta,\lambda)$ is uniformly bounded away from 0 for all sufficiently large $\lambda$ and all $\beta>0$.
\end{proof}

\begin{proof}[Proof of Lemma \ref{UBDecreaseInBeta}]
Let $UB_{n}(n,\lambda)=UB(\frac{n-\lambda}{\sqrt{\lambda}},\lambda)$. Since $n=\lambda+\beta\sqrt{\lambda}$, to show $UB({\beta,\lambda})$ is strictly decreasing in $\beta$ for any $\lambda>0$, it is enough to verify that $UB_{n}(n,\lambda)$ is strictly decreasing in $n$, or equivalently that $$[UB_{n}(n,\lambda)]^{-1}= \frac{\lambda}{n}+\frac{n-\lambda}{\sqrt{n}}\left(\frac{\Phi(a)}{\phi(a)}+\frac{2}{3\sqrt{n}}\right),$$ is strictly increasing in $n$ for any $\lambda$, where $a$ is given in \eqref{eq:alpha}. Now,
\begin{equation}
\footnotesize
\label{eq:partialUBINV}
\frac{\partial [UB_{n}(n,\lambda)]^{-1}}{\partial n}=\frac{-\lambda}{3n^{2}}+\left(\frac{\lambda+n}{2n\sqrt{n}}\right)\left(\frac{\Phi(a)}{\phi(a)}\right)+\left(\frac{n-\lambda}{\sqrt{n}}\right)\left(\frac{\phi(a)+\Phi(a)a}{\phi(a)}\right)\left(\frac{\partial a}{\partial n}\right).
\end{equation}
Here, $$\frac{\partial a}{\partial n}=\frac{-\ln\frac{\lambda}{n}}{a}.$$ Since $n>\lambda$, we have $\frac{\partial a}{\partial n}>0$ and so the third term in \eqref{eq:partialUBINV} is non-negative.
Since $\Phi(a)/\phi(a)$ is strictly increasing in $a$ and $a$ is strictly increasing in $n$, we have that $\Phi(a)/\phi(a)$ is strictly increasing in $n$. Thus the first two terms are greater than $$\frac{-\lambda}{3n^{2}}+\frac{\lambda+n}{2n\sqrt{n}}\frac{\Phi(0)}{\phi(0)},$$ which itself is strictly positive, completing
the proof.
\end{proof}
\begin{proof}[Proof of Lemma \ref{Uniform_Converge_Multi_Term}]
We prove that $$\prod_{i=1}^{L}\left(1-UB( \beta_{i},\lambda_{i}^m)\right)-\prod_{i=1}^{L}\left(1-\tilde{\alpha}\left(\beta_{i},\lambda_{i}^m\right)\right)$$ converges uniformly to 0 in $\beta$ as $m \to \infty$,
via induction. 
Let $\underline{UB}(\beta_i,\lambda_i^m)$ denote $1-UB(\beta_i,\lambda_i^m)$ and let $\underline{\tilde{\alpha}}(\beta_i,\lambda_i^m)$ denote $1-\tilde{\alpha}\left(\beta_{i},\lambda_{i}^m\right)$. We first prove $$\underline{UB}(\beta_{1},\lambda_{1}^m)\underline{UB}(\beta_{2},\lambda_{2}^m)-\underline{\tilde{\alpha}}(\beta_{1},\lambda_{1}^m)\underline{\tilde{\alpha}}(\beta_{2},\lambda_{2}^m)$$ converges uniformly to 0 in $\beta=(\beta_{1},\beta_{2})$ as $m \to \infty$.
From Theorem \ref{UniConvOfBounds}, we have that this result holds separately for $UB(\beta_{1},\lambda_{1}^m)-\tilde{\alpha}(\beta_{1},\lambda_{1}^m)$ and $UB(\beta_{2},\lambda_{2}^m)-\tilde{\alpha}(\beta_{2},\lambda_{2}^m)$, which implies that it again holds separately for $\underline{UB}(\beta_{1},\lambda_{1}^m)-\underline{\tilde{\alpha}}(\beta_{1},\lambda_{1}^m)$ and $\underline{UB}(\beta_{2},\lambda_{2}^m)-\underline{\tilde{\alpha}}(\beta_{2},\lambda_{2}^m)$.
Also, note that
\begin{eqnarray*}
\underline{UB}(\beta_{1},\lambda_{1}^m)\underline{UB}(\beta_{2},\lambda_{2}^m)&-&\underline{\tilde{\alpha}}(\beta_{1},\lambda_{1}^m)\underline{\tilde{\alpha}}(\beta_{2},\lambda_{2}^m)\\
&=&\underline{UB}(\beta_{1},\lambda_{1}^m)\underline{UB}(\beta_{2},\lambda_{2}^m)-\underline{UB}(\beta_{1},\lambda_{1}^m)\underline{\tilde{\alpha}}(\beta_{2},\lambda_{2}^m)\\
&&+\underline{UB}(\beta_{1},\lambda_{1}^m)\underline{\tilde{\alpha}}(\beta_{2},\lambda_{2}^m)-\underline{\tilde{\alpha}}(\beta_{1},\lambda_{1}^m)\underline{\tilde{\alpha}}(\beta_{2},\lambda_{2}^m)\\
&=&\underline{UB}(\beta_{1},\lambda_{1}^m)\left(\underline{UB}(\beta_{2},\lambda_{2}^m)-\underline{\tilde{\alpha}}(\beta_{2},\lambda_{2}^m)\right)\\
&&+\left(\underline{UB}(\beta_{1},\lambda_{1}^m)-\underline{\tilde{\alpha}}(\beta_{1},\lambda_{1}^m)\right)\underline{\tilde{\alpha}}(\beta_{2},\lambda_{2}^m).
\end{eqnarray*}
Thus,
\begin{eqnarray*}
&&\lim_{m\rightarrow\infty}\sup_{\beta\geq0}\left(\underline{UB}(\beta_{1},\lambda_{1}^m)\underline{UB}(\beta_{2},\lambda_{2}^m)-\underline{\tilde{\alpha}}(\beta_{1},\lambda_{1}^m)\underline{\tilde{\alpha}}(\beta_{2},\lambda_{2}^m)\right)\\
&=&\lim_{m\rightarrow\infty}\sup_{\beta\geq0}\left\{\underline{UB}(\beta_{1},\lambda_{1}^m)\left(\underline{UB}(\beta_{2},\lambda_{2}^m)-\underline{\tilde{\alpha}}(\beta_{2},\lambda_{2}^m)\right)\right.\\
&&+\left.\left(\underline{UB}(\beta_{1},\lambda_{1}^m)-\underline{\tilde{\alpha}}(\beta_{1},\lambda_{1}^m)\right)\underline{\tilde{\alpha}}(\beta_{2},\lambda_{2}^m)\right\}\\
&\leq&\lim_{m\rightarrow\infty}\sup_{\beta\geq0}\left\{\underline{UB}(\beta_{1},\lambda_{1}^m)\left(\underline{UB}(\beta_{2},\lambda_{2}^m)-\underline{\tilde{\alpha}}(\beta_{2},\lambda_{2}^m)\right)\right\}\\
&&+\lim_{m\rightarrow\infty}\sup_{\beta\geq0}\left\{\left(\underline{UB}(\beta_{1},\lambda_{1}^m)-\underline{\tilde{\alpha}}(\beta_{1},\lambda_{1}^m)\right)\underline{\tilde{\alpha}}(\beta_{2},\lambda_{2}^m)\right\}.
\end{eqnarray*}
The above inequality holds because both terms are positive. Recall that both the Erlang-C formula and the JVLZ upper bound have range (0, 1]. This immediately implies that $$\lim_{m\rightarrow\infty}\sup_{\beta\geq0}\left\{\underline{UB}(\beta_{1},\lambda_{1}^m)\left(\underline{UB}(\beta_{2},\lambda_{2}^m)-\underline{\tilde{\alpha}}(\beta_{2},\lambda_{2}^m)\right)\right\}=0$$ and $$\lim_{m\rightarrow\infty}\sup_{\beta\geq0}\left\{\left(\underline{UB}(\beta_{1},\lambda_{1}^m)-\underline{\tilde{\alpha}}(\beta_{1},\lambda_{1}^m)\right)\underline{\tilde{\alpha}}(\beta_{2},\lambda_{2}^m)\right\}=0.$$ Thus, we have $$\lim_{m\rightarrow\infty}\sup_{\beta\geq0}\left(\underline{UB}(\beta_{1},\lambda_{1}^m)\underline{UB}(\beta_{2},\lambda_{2}^m)-\underline{\tilde{\alpha}}(\beta_{1},\lambda_{1}^m)\underline{\tilde{\alpha}}(\beta_{2},\lambda_{2}^m)\right)=0.$$ Next, assume that $\forall K\in\{2,3,\ldots,L-1\}$, $$\prod_{i=1}^{K}\left(1-UB( \beta_{i},\lambda_{i}^m)\right)-\prod_{i=1}^{K}\left(1-\tilde{\alpha}\left(\beta_{i},\lambda_{i}^m\right)\right)$$ converges uniformly to 0 in $\beta$ as $m \to \infty$. We now prove that for $K+1$, we have $$\prod_{i=1}^{K+1}\left(1-UB( \beta_{i},\lambda_{i}^m)\right)-\prod_{i=1}^{K+1}\left(1-\tilde{\alpha}\left(\beta_{i},\lambda_{i}^m\right)\right)$$ converges uniformly to 0 in $\beta$ as $m \to \infty$.

As above, we have
\begin{eqnarray*}
&&\left(\prod_{i=1}^{K}\left(1-UB( \beta_{i},\lambda_{i}^m)\right)\right)\underline{UB}(\beta_{K+1},\lambda_{K+1}^m)-\left(\prod_{i=1}^{K}\left(1-\tilde{\alpha}\left(\beta_{i},\lambda_{i}^m\right)\right)\right)\underline{\tilde{\alpha}}(\beta_{K+1},\lambda_{K+1}^m)\\
&=&\left(\prod_{i=1}^{K}\left(1-UB( \beta_{i},\lambda_{i}^m)\right)\right)\underline{UB}(\beta_{K+1},\lambda_{K+1}^m)-\left(\prod_{i=1}^{K}\left(1-UB( \beta_{i},\lambda_{i}^m)\right)\right)\underline{\tilde{\alpha}}(\beta_{K+1},\lambda_{K+1}^m)\\
&&+\left(\prod_{i=1}^{K}\left(1-UB( \beta_{i},\lambda_{i}^m)\right)\right)\underline{\tilde{\alpha}}(\beta_{K+1},\lambda_{K+1}^m)-\left(\prod_{i=1}^{K}\left(1-\tilde{\alpha}\left(\beta_{i},\lambda_{i}^m\right)\right)\right)\underline{\tilde{\alpha}}(\beta_{K+1},\lambda_{K+1}^m)\\
&=&\left(\prod_{i=1}^{K}\left(1-UB( \beta_{i},\lambda_{i}^m)\right)\right)\left(\underline{UB}(\beta_{K+1},\lambda_{K+1}^m)-\underline{\tilde{\alpha}}(\beta_{K+1},\lambda_{K+1}^m)\right)\\
&&+\left(\prod_{i=1}^{K}\left(1-UB( \beta_{i},\lambda_{i}^m)\right)- \prod_{i=1}^{K}\left(1-\tilde{\alpha}\left(\beta_{i},\lambda_{i}^m\right)\right)\right)\underline{\tilde{\alpha}}(\beta_{K+1},\lambda_{K+1}^m),
\end{eqnarray*}
Thus,
\begin{eqnarray*}
&&\lim_{m\rightarrow\infty}\sup_{\beta\geq0}\left\{\prod_{i=1}^{K+1}\left(1-UB( \beta_{i},\lambda_{i}^m)\right)-\prod_{i=1}^{K+1}\left(1-\tilde{\alpha}\left(\beta_{i},\lambda_{i}^m\right)\right)\right\}\\
&\leq&\lim_{m\rightarrow\infty}\sup_{\beta\geq0}\left\{\left(\prod_{i=1}^{K}\left(1-UB( \beta_{i},\lambda_{i}^m)\right)\right)\left(\underline{UB}(\beta_{K+1},\lambda_{K+1}^m)-\underline{\tilde{\alpha}}(\beta_{K+1},\lambda_{K+1}^m)\right)\right\}\\
&&+ \lim_{m\rightarrow\infty}\sup_{\beta\geq0}\left\{\left(\prod_{i=1}^{K}\left(1-UB( \beta_{i},\lambda_{i}^m)\right)- \prod_{i=1}^{K}\left(1-\tilde{\alpha}\left(\beta_{i},\lambda_{i}^m\right)\right)\right)\underline{\tilde{\alpha}}(\beta_{K+1},\lambda_{K+1}^m)\right\}.
\end{eqnarray*}

Again we have that $\left(\prod_{i=1}^{K}\left(1-UB( \beta_{i},\lambda_{i}^m)\right)\right)$ and $\underline{\tilde{\alpha}}(\beta_{K+1},\lambda_{K+1}^m)$ are contained in $(0,1]$. Then by Theorem \ref{UniConvOfBounds} and the induction assumption, we have $$\lim_{m\rightarrow\infty}\sup_{\beta\geq0}\left\{\left(\prod_{i=1}^{K}\left(1-UB( \beta_{i},\lambda_{i}^m)\right)\right)\left(\underline{UB}(\beta_{K+1},\lambda_{K+1}^m)-\underline{\tilde{\alpha}}(\beta_{K+1},\lambda_{K+1}^m)\right)\right\}=0$$ and $$\lim_{m\rightarrow\infty}\sup_{\beta\geq0}\left\{\left(\prod_{i=1}^{K}\left(1-UB( \beta_{i},\lambda_{i}^m)\right)- \prod_{i=1}^{K}\left(1-\tilde{\alpha}\left(\beta_{i},\lambda_{i}^m\right)\right)\right)\underline{\tilde{\alpha}}(\beta_{K+1},\lambda_{K+1}^m)\right\}=0.$$ This implies that $$\prod_{i=1}^{K+1}\left(1-UB( \beta_{i},\lambda_{i}^m)\right)-\prod_{i=1}^{K+1}\left(1-\tilde{\alpha}\left(\beta_{i},\lambda_{i}^m\right)\right)$$ converges uniformly to 0 in $\beta$ as $m \to \infty$, establishing the result.
\end{proof}

\begin{proof}[Proof of Lemma \ref{Det_Single_optimalityResult_extension}]
The inequality $\beta_{\lambda}^{G}\geq\beta_{\lambda}^{F}$ for all $\lambda>0$
is immediate from the definitions of these quantities and the fact that 
$\tilde{\alpha}(\beta,\lambda) \le UB(\beta,\lambda)$ for all $\lambda, \beta>0$. 

Denote the Halfin-Whitt approximation defined in \eqref{HW_2} in Section~\ref{chapter-mathematical-background} as $\alpha_{HW}(\cdot)$, and its inverse function as $\alpha_{HW}^{-1}(\cdot)$. The function $\alpha_{HW}(\cdot)$ is strictly decreasing and this implies that $\alpha_{HW}^{-1}(\cdot)$ is strictly decreasing. For any $\lambda>0$, we denote the inverse of $UB(\cdot,\lambda)$ as $UB^{-1}_{\lambda}(\cdot)$. $UB(\cdot,\lambda)$ is strictly decreasing for any $\lambda>0$, and this implies that $UB^{-1}_{\lambda}(\cdot)$ is strictly decreasing for any $\lambda>0$. By Janssen et al.\ \cite{janssen_08}, we have 
$$\lim_{\lambda\rightarrow\infty}UB(\beta,\lambda)=\alpha_{HW}(\beta), \; \forall \beta>0.$$ 
Together with the monotonicity of $UB(\beta,\lambda)$ in $\lambda$ for any $\beta>0$, we have $$\lim_{\lambda\rightarrow\infty}UB^{-1}_{\lambda}(x)=\alpha^{-1}_{HW}(x), \; \forall x>0.$$ Since $\lim_{\lambda\rightarrow\infty}\epsilon_{\lambda}=\epsilon>0$, there exist $l,u>0$, such that $0<l \leq \epsilon_{\lambda}\leq u$ when $\lambda$ is large enough. Also, since $\alpha_{HW}^{-1}(\cdot)$ is a continuous function, by Theorem \ref{buchanan} we have $$\lim_{\lambda\rightarrow\infty}\sup_{x>0}|UB^{-1}_{\lambda}(x)-\alpha_{HW}^{-1}(x)|=0.$$ This gives $$\lim_{\lambda\rightarrow\infty}UB^{-1}_{\lambda}(\epsilon_{\lambda})=\alpha_{HW}^{-1}(\epsilon),$$ which implies that $\lim_{\lambda\rightarrow\infty} \beta_{\lambda}^{G}$ exists and we denote it by $\beta^{*}$. Then analogous to the proof of Theorem \ref{Det_Single_optimalityResult} in Section~\ref{chapter-deterministic-arrival-rates}, we can show that the limit of $\beta^{F}_{\lambda}$ exists, and $\lim_{\lambda\rightarrow \infty}\beta_{\lambda}^{G}=\lim_{\lambda\rightarrow \infty}\beta_{\lambda}^{F}=\beta^{*}.$
\end{proof}

\end{document}